\documentclass{scrartcl}
\usepackage{ws}

\title{\large \(p\)-SATURATIONS OF WELTER'S GAME \\
 AND THE IRREDUCIBLE REPRESENTATIONS \\
 OF SYMMETRIC GROUPS}
\author{\large Yuki Irie \\
\normalsize Graduate School of Science, Chiba University}
\date{}

\hypersetup{
 pdfauthor={Yuki Irie},
 pdftitle={p-Saturations of Welter's Game and the Irreducible Representations of Symmetric Groups},
 pdfcreator={Emacs 24.5.1 (Org mode 8.3.3)}, 
 pdflang={English}}
\begin{document}

\maketitle

\begin{abstract}
We establish a relation between the Sprague-Grundy function $\sg$ of a \(p\)-saturation of Welter's game and the degrees of the ordinary irreducible representations of symmetric groups.
In \ypass{this game}, a position can be viewed as a partition $\lambda$.
Let $\rho^\lambda$ be the irreducible representation of $\Sym(\size{\lambda})$ indexed by \ypass{$\lambda$}.
For every prime $p$, we show the following results:
(1) the degree of \ypass{$\rho^\lambda$} is prime to \(p\) if and only if $\sg(\lambda) = \size{\lambda}$;
(2) the restriction of \ypass{$\rho^\lambda$} to $\Sym(\sg(\lambda))$ has an irreducible component with degree prime to \(p\).
Further, for every integer $p$ greater than 1, we obtain an explicit formula for $\sg(\lambda)$. 
\end{abstract}

\section{Introduction}
\label{sec:orgheadline9}
\subsection{Welter's game}
\label{sec:orgheadline1}
Welter's game is played with a finite number of coins.
\ypass{These coins} are on a semi-infinite strip of squares numbered \(0, 1, 2, \ldots\)
with no two coins on the same square.
\ypass{This game} has two players.
\ypass{They} alternately move a coin to an empty square with a lower number.
The first player that is not able to move loses.
We now consider the position \(X\) when \ypass{the coins} are on the squares numbered \(x^1, x^2, \ldots, x^\ccm\).
Welter \cite{welter-theory-1954} shows that \ypass{its} Sprague-Grundy number \(\sg(X)\) (see Section \ref{orgtarget1}) 
can be expressed as
\begin{equation}
\label{orgtarget2}
 \sg(X) = x^1 \oplus_2 \cdots \oplus_2 x^m \oplus_2 \left(\bigoplusp[i < j][][2] \fN_2(x^i - x^j) \right),
\end{equation}
where \(\oplus_2\) is binary addition without carry and \(\fN_2(x) = x \oplus_2 (x - 1)\).
Note that \ypass{the position $X$} can be represented both by the \(\ccm\)-set \(\Set{x^1, \ldots, x^\ccm} \subset \NN\)
and by the \(\ccm\)-tuple \((x^1, \ldots, x^\ccm) \in \NN^\ccm\),
where \(\NN\) is the set of all non-negative integers. 
Throughout this paper,
we use \ypass{the set representation}, unless otherwise specified.

\ypass{Welter's game} can also be played with a Young diagram \cite{sato-game-1968}.
Let \(\sigma\) be the permutation of \(\set{1, 2, \ldots, \ccm}\) such that  \(x^{\sigma(1)} > x^{\sigma(2)} > \cdots > x^{\sigma(\ccm)}\).
Let \(\lambda(X)\) be the partition \((x^{\sigma(1)} - \ccm + 1, x^{\sigma(2)} - \ccm + 2, \ldots, x^{\sigma(\ccm)})\).
We identify \ypass{$\lambda(X)$} with its Young diagram 
\[
\Set{ (i, j) \in \ZZ^2 : 1 \le i \le \ccm,\; 1 \le j \le x^{\sigma(i)} - \ccm + i}.
\]
As a result, moving a coin corresponds to removing a hook.
Sato \cite{sato-game-1968, sato-Mathematical-1970, sato-maya-1970} obtains (\ref{orgtarget2}) independently.
In addition, he finds that \(\sg(X)\) can be expressed in a form
similar to the hook-length formula for the degrees of the irreducible representations of symmetric groups.
Kawanaka \cite{kawanaka-GAMES-} points out that
\(\sg(X)\) can also be expressed using the 2-core tower of \(\lambda(X)\).\footnote{Kawanaka \cite{kawanaka-GAMES-} also generalizes Welter's game and obtains an efficient algorithm to find winning moves.}

The purpose of this paper is to establish a relation between 
the Sprague-Grundy function of a \(p\)-saturation of Welter's game
and the degrees of the irreducible representations of symmetric groups.

\subsection{\(p\)-Saturations}
\label{sec:orgheadline2}
Let \(p\) be an integer greater than 1.
For each \(\ccm \in \NN\), let \(\cW^\ccm\) denote \ypass{Welter's game} with \(\ccm\) coins.

To define \ypass{\(p\)-saturations} of \(\cW^m\), 
we introduce a game \(\cW^\ccm_{p, \cck}\) called \(\cW^m\) with \(p\)-index \(\cck\),
where \(\cck\) is a positive integer.
This game comes from Moore's \(\text{Nim}_\cck\) (Nim with index \(\cck\)) \cite{moore-generalization-1910} and Flanigan's \(\text{Rim}_\cck\).\footnote{Players can move at most \(k - 1\) coins in \(\text{Nim}_{\cck}\) and \(\text{Rim}_{\cck}\).
\(\text{Rim}_\cck\) is introduced  in an unpublished paper \cite{flanigan-Nim-1980}.
In Section 3, we show that \(p\)-saturations of Nim have the same Sprague-Grundy function as \(\text{Rim}_p\).}
A position in \(\cW^\ccm_{p, \cck}\) is the same as in \(\cW^\ccm\).
Let \(X\) and \(Y\) be two positions \((x^1, \ldots, x^\ccm)\) and \((y^1, \ldots, y^\ccm)\) in \(\cW^\ccm\).
Let \(\dist(X, Y)\) denote the Hamming distance between \(X\) and \(Y\), that is, \(\dist(X, Y) = \ssize{\set{i : 1 \le i \le \ccm,\, x^i \neq y^i}}\).
In \(\cW^\ccm_{p, \cck}\), players can move from \(X\) to \(Y\) if and only if 
\begin{enumerate}
\item \(0 < \dist(X, Y) < \cck\),
\item \(y^i \le x^i\) for \(1 \le i \le \ccm\),
\item \(\displaystyle\ord\left(\sum_{i = 1}^{\ccm} x^i - y^i\right) = \min \Set{\ord(x^i - y^i) : 1 \le i \le \ccm}\),
\end{enumerate}
where \(\ord(x)\) is the \(p\)-adic order of \(x\),
that is, \(\ord(x) = \max \set{L \in \NN : p^L \mid x}\) if \(x \neq 0\), and \(\ord(0) = \infty\).
For example, in \(\cW_{2, 3}^2\), they can move from \((4, 3)\) to \((0, 1)\), but cannot move from \((4, 3)\) to \((1, 0)\).
In \ypass{the set representation}, they can move from \(\Set{3, 4}\) to \(\Set{0, 1}\).

\ypass{The game} \(\cW_{p, k}^\ccm\) is called a \emph{\(p\)-saturation} of \(\cW^\ccm\) if 
it has the same Sprague-Grundy function as \(\cW_{p, \ccm + 1}^\ccm\).
The smallest such \(\cck\) is called the \emph{\(p\)-saturation index} of \(\cW^\ccm\),
and denoted by \(\sat_p(\cW^\ccm)\).
By definition, if \(\cch \ge \sat_p(\cW^\ccm)\), then
\(\cW_{p, h}^\ccm\) also has the same Sprague-Grundy function as \(\cW_{p, \ccm + 1}^\ccm\).
\ypass{The formula} (\ref{orgtarget2}) implies that \(\sat_2(\cW^\ccm) = 2\) for every positive integer \(\ccm\).
In general, we can show that \(\sat_p(\cW^\ccm)\) is at least \(\min(p, \ccm + 1)\),
but we do not know its exact value for \(3 \le p \le \ccm\) (see Section \ref{orgtarget3}).

\subsection{Main results}
\label{sec:orgheadline7}
\label{orgtarget4}

We begin by introducing some definitions.
Let \(X\) be a position in \ypass{Welter's game}, and let \(L\) be a non-negative integer.
For each \((i, j) \in \lambda(X)\), the \emph{hook}  \(H_{i, j}(X)\) is defined by
\[
 H_{i, j}(X) = \Set{(i', j') \in \lambda(X) : (i' \ge i \tand j' = j) \tor  (i' = i \tand j' \ge j)}.
\]
\ypass{A hook} is called a \emph{\(p^L\)-hook} if its length (size) is a multiple of \(p^L\).
Let \(\btower{X}[L]\) denote the remainder of the number of \ypass{\(p^L\)-hooks} in \(\lambda(X)\) divided by \(p\).
We define
\begin{equation}
\label{orgtarget5}
  \btower{X} = \sum_{L \in \NN} \btower{X}[L] p^L.
\end{equation}
Let \(\oplus_p\) and \(\ominus_p\) be \(p\)-ary addition and subtraction without carry, respectively.
For each \(x \in \ZZ\), let \(\fN_p(x) = x \ominus_p (x - 1)\) .
\begin{thm}
 \comment{Thm.}
\label{sec:orgheadline3}
\label{orgtarget6}
Let \(X\) be a position \(\set{x^1, \ldots, x^\ccm}\) in a \(p\)-saturation of \(\cW^\ccm\).
Then there exists a position \(Y\) such that \(\lambda(Y) \subset \lambda(X)\) and \(\size{\lambda(Y)} = \sg(Y) = \sg(X)\).
Further, 
\begin{equation}
\label{orgtarget7}
\begin{split}
  \sg(X) &= \btower{X}\\
  &= x^1 \oplus_p \cdots \oplus_p x^\ccm \ominus_{p} \left (\bigoplusp[i < j][][p] \fN_p(x^i - x^j) \right) \\
  &= \bigopluspp[(i, j) \in \lambda(X)][][p] \fN_p\left(\size{H_{i, j}(X)}\right).
\end{split}
\end{equation}
 
\end{thm}

\comment{Macdonald's result}
\label{sec:orgheadline4}
\label{orgtarget8}
Suppose that \(p\) is a prime.
Let \(X\) be a position in \ypass{a \(p\)-saturation} of Welter's game,
and let \(\rho^X\) be the irreducible representation of the symmetric group \(\Sym(\size{\lambda(X)})\) indexed by \(\lambda(X)\).
By Macdonald's result \cite{macdonald-degrees-1971},
we see that
the degree of \ypass{$\rho^X$} is prime to \(p\)
if and only if \(\btower{X} = \size{\lambda(X)}\). 
From Theorem \ref{orgtarget6} we obtain the following corollary.
\begin{cor}
 \comment{Cor.}
\label{sec:orgheadline5}
\label{orgtarget9}
Let \(X\) be a position in a \(p\)-saturation of Welter's game.
If \(p\) is a prime, then the following assertions hold:
\begin{enumerate}
\renewcommand{\labelenumi}{\textnormal{(\arabic{enumi})}}
\item The degree of $\rho^X$ is prime to $p$ if and only if $\sg(X) = \size{\lambda(X)}$.
\item The restriction of $\rho^X$ to $\Sym(\sg(X))$ has an irreducible component with degree prime to \(p\), where $\Sym(0) = \Sym(1)$.
\end{enumerate}
 
\end{cor}

\comment{Connect}
\label{sec:orgheadline6}
Let \(X\) be a position in a \(p\)-saturation of Welter's game.
In view of Corollary \ref{orgtarget9},
it is natural to consider the maximum Sprague-Grundy number of a position \(Y\) with \(\lambda(Y) \subset \lambda(X)\) and \(\sg(Y) = \size{\lambda(Y)}\).
A lower bound for this value is presented in Proposition \ref{orgtarget10}.

\subsection{Organization}
\label{sec:orgheadline8}
This paper is organized as follows.
In Section 2, we introduce a notation
and recall the concepts of impartial games and \(p\)-core towers.
Section 3 contains the basic properties of \(p\)-saturations of Nim and Welter's game.
In Section 4, we introduce \(p^\ccH\)-options and reduce Theorem \ref{orgtarget6} to three technical lemmas on \(p^H\)-options.
These lemmas are proven in Sections 5 and 6.
\section{Preliminaries}
\label{sec:orgheadline17}
\label{orgtarget1}
Throughout this paper, \(p\) is an integer greater than 1.
We write \(\oplus\) instead of \(\oplus_p\).

\subsection{Notation}
\label{sec:orgheadline10}
Let \(L \in \NN\).
Let \(\ZZ_{p^L}\) be the ring of integers modulo \(p^L\).
We regard \(\ZZ_{p^L}\) as \(\set{0, 1, \ldots, p^{L - 1}}\).

Let \(x\) and \(y\) be two integers.
Let \(\pexp{x}[L][p]\) denote the \(L\)th digit in the \(p\)-adic expansion of \(x\).
We write \(\pexp{x}[L]\) instead of \(\pexp{x}[L][p]\) when no confusion can arise.
Thus \(x = \sum_{L \in \NN} \pexp{x}[L] p^L\).
We identify \(x \in \ZZ\) with the infinite sequence \((\pexp{x}[0], \pexp{x}[1], \ldots) \in \ZZop^\NN\).
For clarity, we sometimes write \((\pexp{x}[0], \pexp{x}[1], \ldots)_{p}\) instead of \((\pexp{x}[0], \pexp{x}[1], \ldots)\).
In this notation,
\[
 x \oplus y = (\pexp{x}[0] + \pexp{y}[0], \pexp{x}[1] + \pexp{y}[1], \ldots)
 \quad \tand \quad 
 x \ominus y = (\pexp{x}[0] - \pexp{y}[0], \pexp{x}[1] - \pexp{y}[1], \ldots).
\]
Let \(\pexp{x}[\ge L]\) denote the quotient of \(x\) divided by \(p^L\), that is, \(\pexp{x}[\ge L] = (\pexp{x}[L], \pexp{x}[L + 1], \ldots)\).
Let \(\pexp{x}[<L]\) denote the residue of \(x\) modulo \(p^L\).
We identify \(\pexp{x}[<L] \in \ZZ_{p^L}\) with the finite sequence \((\pexp{x}[0], \pexp{x}[1], \ldots, \pexp{x}[L - 1]) \in \ZZop[L]\).

For each \(* \in \Set{+, -, \oplus, \ominus}\),
we define \(\pexp{x}[<L]\, *\, y\) by
\[
 \pexp{x}[<L]\, *\, y = \pexp{(x\, *\, y)}[<L] \in \ZZop[L],
\]
which is well-defined.
For example, 
\((0, 0) \ominus 1 = (p - 1, 0)\), while \((0, 0) - 1 = (p - 1, p - 1)\) and \((0, 0) \ominus p^2 = (0, 0)\).

For each \(r \in \ZZop\), let \((\pexp{x}[<L], \ccr)\) denote \((\pexp{x}[0], \pexp{x}[1], \ldots, \pexp{x}[L - 1], \ccr) \in \ZZop[L + 1]\).

\subsection{Games}
\label{sec:orgheadline15}
Let \(\Gamma\) be a directed graph \((\cV, \cA)\), that is, \(\cV\) is a set and \(\cA \subset \cV \times \cV\).
For \(X \in \cV\), let \(\lg(X)\) denote the maximum length of a path from \(X\).
In this paper, a directed graph \(\Gamma\) is called an \emph{(impartial) game}
if \(\lg(X)\) is finite for every vertex \(X\) of \(\Gamma\).
Let \(\Gamma\) be a game.
A vertex of \ypass{$\Gamma$} is called a \emph{position} in \(\Gamma\).
For two positions \(X\) and \(Y\) in \(\Gamma\),
the position \(Y\) is called an \emph{option} of \(X\) if there exists an edge from \(X\) to \(Y\),
and a \emph{descendant} of \(X\) if there exists a path from \(X\) to \(Y\).
For example, in Welter's game, \(Y\) is a descendant of \(X\) if and only if \(\lambda(Y) \subset \lambda(X)\).
A descendant \(Y\) of \(X\) is said to be \emph{proper} if \(Y \neq X\).

\begin{exm}[Nim with \(p\)-index \(k\)]
 \comment{Exm. [Nim with \(p\)-index \(k\)]}
\label{sec:orgheadline11}
\label{orgtarget11}
Recall that Nim is Welter's game without the restriction that the coins are on distinct squares.
For each \(\ccm \in \NN\), let \(\cN^\ccm\) denote Nim with \(\ccm\) coins.
Let \(\cD^\ccm\) be the set of all \(m\)-tuples \((d^1, \ldots, d^\ccm) \in \NN^\ccm\)
such that
\begin{equation}
\label{orgtarget12}
 \ord\left(\sum_{i = 1}^\ccm d^i\right) = \min \Set{\ord(d^i) : 1 \le i \le \ccm}.
\end{equation}
Let \(\cV = \NN^\ccm\). For each positive integer \(k\), let \(\cA_k =  \sset{(X, Y) \in \cV^2 : X - Y \in \cD^\ccm,\, 0 < \dist(X, Y) < k}\).
The game \((\cV, \cA_k)\) is called \emph{\(\cN^\ccm\) with \(p\)-index \(\cck\)}, and denoted by \(\cN^\ccm_{p, \cck}\).
Note that \(\cN^\ccm_{p, 2} = \cN^\ccm\).
 
\end{exm}
\comment{Comment}
\label{sec:orgheadline12}
\begin{exm}[Welter's game with \(p\)-index \(k\)]
 \comment{Exm. [Welter's game with \(p\)-index \(k\)]}
\label{sec:orgheadline13}
Let
\[
 \cV = \Set{(x^1, \ldots, x^\ccm) \in \NN^\ccm : x^i \neq x^j \; \tfor \; 1 \le i < j \le \ccm}.
\]
The induced subgraph of \(\cN^\ccm_{p, \cck}\) to \(\cV\) is \(\cW^\ccm_{p, \cck}\) in the tuple representation. 
This implies that, for every induced subgraph \(\Gamma\) of \(\cN^\ccm\), we can define the \(p\)-saturations of \(\Gamma\).
 
\end{exm}

\comment{Connect Sprague-Grundy}
\label{sec:orgheadline14}
We now introduce Sprague-Grundy functions.
Let \(\Gamma\) be a game and \(X\) a position in \(\Gamma\).
The \emph{Sprague-Grundy number} \(\sg(X)\) of \(X\) is defined as the minimum non-negative integer \(n\)
such that \(X\) has no option \(Y\) with \(\sg(Y) = n\).
\ypass{This function} \(\sg\) is called the \emph{Sprague-Grundy function} of \(\Gamma\).
The Sprague-Grundy theorem \cite{grundy-mathematics-1939, sprague-uber-1935} states that \(X\) is equivalent to
the position \((\sg(X))\) in \(\cN^1\).
In particular, \(X\) is a winning position for the previous player if and only if \(\sg(X) = 0\).
See \cite{berlekamp-Winning-2001, conway-numbers-2001} for details.
Note that by definition, \(\sg(X) \le \lg(X)\). For example, in Welter's game, \(\sg(X) \le \lg(X) = \ssize{\lambda(X)}\).

\subsection{\(p\)-Core towers}
\label{sec:orgheadline16}
We define \(p\)-core towers and state their properties.
See \cite{olsson-combinatorics-1993} for details.
Let \(X\) be a position in \ypass{Welter's game}.
Let \(L\) be a non-negative integer.

For each \(\ccR \in \ZZop[L]\), define
\begin{equation}
\label{orgtarget13}
 X_\ccR = \Set{ \pexp{x}[\ge L] : x \in X,\quad \pexp{x}[<L] = \ccR}.
\end{equation}
\ypass{The position} \(X\) is uniquely determined by \((X_\ccR)_{\ccR \in \ZZop[L]}\). Indeed,
let
\begin{equation}
\label{orgtarget14}
[X_\ccR]_{\ccR \in \ZZop[L]} = \Set{(\ccR, x) : \ccR \in \ZZ_p^L,\quad x \in X_R},
\end{equation}
where \((\ccR, x) = (\ccR_0, \ccR_1, \ldots, \ccR_{L - 1}, \pexp{x}[0], \pexp{x}[1], \ldots)\).
Then \(X = [X_\ccR]_{\ccR \in \ZZop[L]}\).
With this notation, we define the \emph{\(p\)-core} \(X_{(p)}\) of \ypass{$X$} by
\begin{equation}
\label{orgtarget15}
 X_{(p)} = \left[\Set{0,1, \ldots, \size{X_\ccr} - 1}\right]_{\ccr \in \ZZop}.
\end{equation}
By definition, \(\lambda(X_{(p)})\) is the partition obtained by removing all \(p\)-hooks from \(\lambda(X)\).
\ypass{The sequence} \(((X_\ccR)_{(p)})_{\ccR \in \ZZop[L]}\) is called \emph{the \(L\)th row of the \(p\)-core tower} of \(X\).
We define
\begin{equation}
\label{orgtarget16}
 \tower{X}[L] = \sum_{\ccR \in \ZZop[L]} \size{\lambda\left((X_{\ccR})_{(p)}\right)}.
\end{equation}
For example, if \(x \in \NN\), then \(\tower{\set{x}}[L] = \pexp{x}[L]\).
Let \(\tower{X}\) denote the sequence whose \(L\)th term is \(\tower{X}[L]\).

Recall that there exists a canonical bijection between the set of all \(p\)-hooks
in \(\lambda(X)\) and the set of all hooks in \(\lambda(X_0), \ldots, \lambda(X_{p - 1})\).
Let \(\hook{X}[L]\) be the number of \(p^L\)-hooks in \(\lambda(X)\)
and \(\hook{X}\) the sequence whose \(L\)th term is \(\hook{X}[L]\). 
Then
\begin{equation}
\label{orgtarget17}
\begin{split}
 \hook{X}[L] = \sum_{\ccR \in \ZZop[L]} \hook{X_\ccR}[0] &= \sum_{\ccR \in \ZZop[L]} \left(\sum_{x \in X_\ccR} x  - \left(0 + 1 + \cdots + (\size{X_{\ccR}} - 1)\right) \right)\\
 &= \sum_{x \in X} \pexp{x}[\ge L] - \sum_{\ccR \in \ZZop[L]} {\size{X_\ccR} \choose 2}.
\end{split}
\end{equation}
Furthermore, \(\tower{X}\) satisfies the following properties:
\begin{equation}
\label{orgtarget18}
 \tower{X}[L] = \hook{X}[L] - p \hook{X}[L + 1],
\end{equation}
\begin{equation}
\label{orgtarget19}
 \sum_{L \in \NN}  \tower{X}[L] p^L = \size{\lambda(X)},
\end{equation}
\begin{equation}
\label{orgtarget20}
 \tower{X}[\ge L] = \sum_{\ccR \in \ZZop[L]} \tower{X_\ccR},
\end{equation}
where \(\tower{X}[\ge L] = (\tower{X}[L], \tower{X}[L + 1], \ldots)\).
Note that by (\ref{orgtarget18}), \(\btower{X}[L]\) is equal to the remainder of \(\tower{X}[L]\) divided by \(p\).
In particular, it follows from  (\ref{orgtarget19}) that \(\btower{X} = \ssize{\lambda(X)}\) if and only if
\(\btower{X}= \tower{X}\), that is, \(\btower{X}[L] = \tower{X}[L]\) for each \(L \in \NN\).
\section{\(p\)-Saturations}
\label{sec:orgheadline33}
\label{orgtarget3}
We present the basic properties of \(p\)-saturations of Nim and Welter's game.

\subsection{\(p\)-Saturations of Nim}
\label{sec:orgheadline22}
\comment{Intro Nim}
\label{sec:orgheadline18}
Let \(X\) be a position \((x^1, \cdots, x^\ccm)\) in a \(p\)-saturation of \(\cN^\ccm\).
We show that \(\sg(X) = x^1 \oplus \cdots \oplus x^\ccm\)
and \(\sat_p(\cN^\ccm)= \min(p, \ccm + 1)\).

\begin{lem}
 \comment{Lem. p-saturation of Nim}
\label{sec:orgheadline19}
\label{orgtarget21}
Let \(X\) be a position \((x^1, \ldots, x^\ccm)\) in \(\cN^\ccm_{p, \cck}\).
Then the following assertions hold:
\begin{enumerate}
\renewcommand{\labelenumi}{\textnormal{(\arabic{enumi})}}
\item Let $Y$ be a proper descendant $(y^1, \ldots, y^\ccm)$ of $X$ such that $\dist(X, Y) < \cck$. Then
 \[
 \ord \left(\bigoplus_{i = 1}^\ccm x^i - \bigoplus_{i = 1}^\ccm y^i \right) \ge \min \set{\ord(x^i - y^i) : 1 \le i \le \ccm}
 \]
 with equality if and only if $Y$ is an option of $X$.
\item Suppose that $k \ge \min(p, \ccm + 1)$. For $0 \le \cch <  x^1 \oplus \cdots \oplus x^\ccm$, the position $X$ has an option $(y^1, \ldots, y^\ccm)$ with $y^1 \oplus \cdots \oplus y^\ccm = \cch$. In particular, $\sg(X) = x^1 \oplus \cdots \oplus x^m$.
\end{enumerate}
 
\end{lem}

\begin{proof}
 \comment{Proof. Lem. p-saturation of Nim}
\label{sec:orgheadline20}
(1)\,
Let \(N = \min \set{\ord(x^i - y^i) : 1 \le i \le \ccm }\).
Then \(\ord(\bigoplus_{i = 1}^\ccm x^i - \bigoplus_{i = 1}^\ccm y^i) \ge N\) and
\[
 \Pexp{\left(\bigoplus_{i = 1}^\ccm x^i - \bigoplus_{i = 1}^\ccm y^i\right)}[N] =  \Pexp{\left(\sum_{i = 1}^\ccm x^i - \sum_{i = 1}^\ccm y^i\right)}[N].
\]
Therefore \(Y\) is an option of \(X\)
if and only if \(\ord(\bigoplus_{i = 1}^\ccm x^i - \bigoplus_{i = 1}^\ccm y^i) = N\).

\vspace{0.5em}
\noindent
(2)\,
Let \(\ccn = x^1 \oplus \cdots \oplus x^\ccm\) and \(N = \max \Set{L \in \NN : \pexp{\ccn}[L] \neq \pexp{\cch}[L] }\). 
Since \(\cch < \ccn\), it follows that
\(\pexp{\cch}[N] < \pexp{\ccn}[N] = \pexp{x^1}[N] + \cdots + \pexp{x^\ccm}[N]\).
Thus there exist \(\ccr^1, \ldots, \ccr^\ccm \in \ZZ_p\) such that
\[
 \sum_{i = 1}^\ccm \ccr^i = \pexp{\cch}[N] \quad \tand \quad \ccr^i \le \pexp{x^i}[N] \quad \tfor 1 \le i \le \ccm.
\]
We may assume that \(\ccr^1 < \pexp{x^1}[N]\). 
Since \(\pexp{\ccn}[N] \le p - 1\), we may also assume that 
\[
 \dist\left((\ccr^1, \ldots, \ccr^\ccm),\, (\pexp{x^1}[N], \ldots, \pexp{x^\ccm}[N])\right) < k.
\]
Let 
\[
 y^1 = \left(\pexp{x^1}[0] - \pexp{\ccn}[0] + \pexp{\cch}[0], \ldots, \pexp{x^1}[N - 1] - \pexp{\ccn}[N - 1] + \pexp{\cch}[N - 1], \ccr^1, \pexp{x^1}[N + 1], \pexp{x^1}[N + 2], \ldots\right)_p,
\]
\[
 y^i = \left(\pexp{x^i}[0], \ldots, \pexp{x^i}[N - 1], \ccr^i, \pexp{x^i}[N + 1], \pexp{x^i}[N + 2] \ldots\right)_p \quad \tfor 2 \le i \le \ccm,
\]
and \(Y = (y^1, \ldots, y^\ccm)\).
Then \(Y\) is a proper descendant of \(X\) and \(y^1 \oplus \cdots \oplus y^\ccm = \cch\).
Since \(\dist(X, Y) < k\) and
\[
 \ord\left(\bigoplus_{i = 1}^\ccm x^i - \bigoplus_{i = 1}^\ccm y^i\right) = \ord(x^1 - y^1) = \min \Set{\ord(x^i - y^i) : 1 \le i \le \ccm},
\]
the position \(Y\) is an option of \(X\) by (1).

Since \(X\) has no option \((z^1, \ldots, z^\ccm)\) with \(z^1 \oplus \cdots \oplus z^\ccm = x^1 \oplus \cdots \oplus x^\ccm\) by (1),
we obtain \(\sg(X) = x^1 \oplus \cdots \oplus x^\ccm\).
 
\end{proof}

\comment{The p-saturation index of Nim}
\label{sec:orgheadline21}
It remains to show that \(\sat_p(\cN^\ccm) = \min(p, \ccm + 1)\).
Let \(k = \min(p, \ccm + 1)\). 
If \(\ccm = 0\), then it is clear. Suppose that \(\ccm \ge 1\). Then \(k \ge 2\). Let
\[
 X = (\underbrace{p, \ldots, p}_{k - 1}, 0, \ldots, 0) \in \NN^{\ccm}.
\]
The position \(X\) has no option \((y^1, \ldots, y^{\ccm})\) with \(y^1 \oplus \cdots \oplus y^\ccm = 0\) in \(\cN^{\ccm}_{p, k - 1}\).
Thus \(\sat_p(\cN^\ccm) = k\) by Lemma \ref{orgtarget21}.

\subsection{\(p\)-Saturations of Welter's game}
\label{sec:orgheadline32}
\comment{Intro Welter}
\label{sec:orgheadline23}
Let \(X\) be a position in a \(p\)-saturation of \(\cW^\ccm\).
We define
\[
 \msg(X) = \max \Set{\sg(Y) : Y \text{\ is a descendant of\ } X \text{\ with\ } \sg(Y) = \size{\lambda(Y)}}.
\]
We give lower bounds for \(\msg(X)\) and \(\sat_p(\cW^\ccm)\).
We also show that the right-hand sides of (\ref{orgtarget7}) are equal.

\comment{p-Saturation index of Welter}
\label{sec:orgheadline24}
The next lemma provides a necessary and sufficient condition for a descendant to be an option in the tuple representation.
\begin{lem}
 \comment{Lem. option-condition}
\label{sec:orgheadline25}
\label{orgtarget22}
Let \(X\) be a position \((x^1, \ldots, x^\ccm)\) in \(\cW^\ccm_{p, k}\)
and \(Y\) its proper descendant \((y^1, \ldots, y^\ccm)\) such that \(\dist(X, Y) < \cck\).
Then 
 \[
  \ord(\btower{X} - \btower{Y}) \ge \min \set{\ord(x^i - y^i) : 1 \le i \le \ccm}
\]
with equality if and only if \(Y\) is an option of \(X\).
 
\end{lem}

\begin{proof}
 \comment{Proof.}
\label{sec:orgheadline26}
Let \(N = \min \set{\ord(x^i - y^i) : 1 \le i \le \ccm}\).
Then \(\ssize{X_\ccR} = \ssize{Y_\ccR}\) for each \(\ccR \in \ZZ_p^N\). Thus \(\tower{X}[< N] = \tower{Y}[< N]\),
where \(\tower{X}[< N] = (\tower{X}[0], \tower{X}[1], \ldots, \tower{X}[N - 1])\).
This shows that \(\ord(\btower{X} - \btower{Y}) \ge N\).
By (\ref{orgtarget17}), we have
\[
 \hook{X}[N] - \hook{Y}[N] = \sum_{i = 1}^\ccm \pexp{x^i}[\ge N] - \pexp{y^i}[\ge N] =
 \Pexp{\left(\sum_{i = 1}^\ccm p^N \pexp{x^i}[\ge N] - p^N \pexp{y^i}[\ge N]\right)}[\ge N] = \Pexp{\left(\sum_{i = 1}^\ccm x^i - y^i\right)}[\ge N],
\]
and hence
\[
 \btower{X}[N] - \btower{Y}[N] \equiv \hook{X}[N] - \hook{Y}[N] \equiv \Pexp{\left(\sum_{i = 1}^\ccm x^i - y^i\right)}[N]  \pmod{p}.
\]
Therefore \(\ord(\btower{X} - \btower{Y}) = N\) if and only if \(Y\) is an option of \(X\).
 
\end{proof}

\comment{Conenct}
\label{sec:orgheadline27}
We now show that \(\sat_p(\cW^\ccm) \ge \min(p, \ccm + 1)\).
If \(\ccm = 0\), then it is clear.
Suppose that \(\ccm \ge 1\).
For each \(\ccn \in \NN\), we define 
\begin{equation}
\label{orgtarget23}
\pshift{X}{\ccn} = \Set{0,1, \ldots, \ccn - 1} \cup \Set{x + \ccn : x \in X}.
\end{equation}
By definition, \(\lambda(\pshift{X}{\ccn}) = \lambda(X)\),
and so \ypass{the two sets} \(X\) and \(\pshift{X}{\ccn}\) represent essentially the same position.
Let \(k = \min(p, \ccm + 1)\) and \(X = \pshift{\Set{p, p + 1, \ldots, p + k - 2}}{\ccm - k + 1}\).
Then \(k \ge 2\) and \(\btower{X} = p(k - 1)\), but Lemma \ref{orgtarget22} implies that \(X\) has no option \(Y\) with \(\btower{Y} = 0\)
in \(\cW^\ccm_{p, k - 1}\).
Hence \(\sat_p(\cW^\ccm) \ge k\) by Theorem \ref{orgtarget6}.

\comment{A lower bound for \(\msg(X)\) if \(X_N > p\)}
\label{sec:orgheadline28}
We next turn to \(\msg(X)\).
\ypass{Theorem} \ref{orgtarget6} asserts that \(\msg(X) \ge \sg(X)\).
The following result improves \ypass{this} bound.
\begin{prop}
 \comment{Prop. lower-bound}
\label{sec:orgheadline29}
\label{orgtarget10}
Let \(X\) be a position in a \(p\)-saturation of Welter's game.
If \(\tower{X}[N] \ge p + 1\) for some \(N \in \NN\),
then
\begin{equation}
\label{orgtarget24}
 \msg(X) \ge (\underbrace{p - 1, \ldots, p - 1}_{N + 1}, \btower{X}[N + 1], \btower{X}[N + 2], \ldots)_p.
\end{equation}
 
\end{prop}
\comment{Postpone the proof}
\label{sec:orgheadline30}
We postpone the proof of this result to Section \ref{orgtarget25}.

Let \(X\) be a position in a \(p\)-saturation of \(\cW^\ccm\) such that \(\tower{X}[N] = p\) for some \(N \in \NN\).
Then \(X\) does not necessarily satisfy (\ref{orgtarget24}).
For example, let \(p = 3\), \(X = \set{3, 7}\), and \(Y = \set{3, 4, 5}\). Then \(\tower{X} = \tower{Y} = (0, 3, 0, \ldots)\).
However, \(\msg(X) = (2, 2, 0, \ldots)\) and \(\msg(Y) = (0, 2, 0, \ldots)\).
This is because \(X\) has a \(p^0\)-option (see the next section), while \(Y\) does not.
A sufficient condition for a position to have a \(p^0\)-option is given in Lemma \ref{orgtarget26} below.

\comment{Three expressions for \(sg(X)\)}
\label{sec:orgheadline31}
Let \(X = \Set{x^1, \ldots, x^m}\). We close this section by proving
\[
 \btower{X}= x^1 \oplus \cdots \oplus x^\ccm \ominus \left (\bigoplusp_{i < j} \fN_p(x^i - x^j) \right) = \bigoplus_{(i, j) \in \lambda(X)} \fN_p\left(\size{H_{i, j}(X)}\right).
\]
We may assume that \(x^1 > \cdots > x^m\).
Since \(\fN_p(x) = x \ominus (x - 1) = \sum_{L = 0}^{\ord(x)} p^L\), it follows that
\[
 \Pexp{\left(\bigoplus_{(i, j) \in \lambda(X)} \fN_p\left(\size{H_{i, j}(X)}\right)\right)}[\hspace{0.3em} L] \equiv \hook{X}[L] \pmod{p}.
\]
Hence \(\bigoplus_{(i, j) \in \lambda(X)} \fN_p(\ssize{H_{i, j}(X)}) = \btower{X}\).
We also have
\begin{align*}
 \bigoplus_{(i, j) \in \lambda(X)} \fN_p\left(\size{H_{i, j}(X)}\right) &= \bigoplus_{i = 1}^m \left(\bigoplus_{0 \le y < x^i,\,\, y \not \in X} \fN_p(x^i - y)\right) \\
 &= \bigoplus_{i = 1}^m \left( \bigoplus_{0 \le y < x^i} \fN_p(x^i - y) \ominus \left(\bigoplus_{i < j} \fN_p(x^i - x^j)\right) \right)\\
 &= x^1 \oplus \cdots \oplus x^m \ominus \left(\bigoplus_{i < j} \fN_p(x^i - x^j)\right).
\end{align*}
\section{Proof of the main results}
\label{sec:orgheadline84}
\label{orgtarget25}
We reduce Theorem \ref{orgtarget6} and Proposition \ref{orgtarget10} to three technical lemmas.
In the rest of the paper, we call a position in \(\cW^{\ccm}_{p, \ccm + 1}\) simply a position.

\subsection{Sprague-Grundy numbers}
\label{sec:orgheadline43}
\comment{Theorem 1.1'}
\label{sec:orgheadline34}
\label{orgtarget27}
Theorem \ref{orgtarget6} follows immediately from the next result.

\comment{Thm  main2}
\label{sec:orgheadline35}
\vspace{1em}
\noindent
\textbf{Theorem 1.1.\hspace{-0.2em}'\(\;\)}
\label{orgtarget28}
\textit{
Let \(X\) be a position. 
Then the following assertions hold:
\begin{description}
\item[{\textnormal{(A1)}}] If \(\btower{X} = \size{\lambda(X)} > 0\), then \(X\) has a descendant \(Y\) with \(\btower{Y} = \btower{X} - 1\).
\item[{\textnormal{(A2)}}] If \(\btower{X} < \size{\lambda(X)}\), then \(X\) has a proper descendant \(Y\) with \(\btower{Y} \ge \btower{X}\).
\end{description}
In particular, \(\sg(X) = \btower{X}\).
\vspace{1em}
}

\comment{Show by induction}
\label{sec:orgheadline36}
Let \(X\) be a position. 
We prove Theorem 1.1' by induction on \(\size{\lambda(X)}\).
If \(\size{\lambda(X)} = 0\), then it is clear.
Suppose that \(\size{\lambda(X)} > 0\).

\comment{Show how (A1) and (A2) imply the Sprague-Grundy formula}
\label{sec:orgheadline37}
The assertions (A1) and (A2) are proven in Subsections \ref{orgtarget29} and \ref{orgtarget30}, respectively.
In this subsection, we show how (A1) and (A2) imply that \(\sg(X) = \btower{X}\).
Since \(X\) has no option \(Y\) with \(\btower{Y} = \btower{X}\) by Lemma \ref{orgtarget22},
it suffices to show that
\begin{description}
\item[{\textnormal{(SG)}}] if \(0 \le \cch < \btower{X}\), then \(X\) has an option \(Y\) with \(\btower{Y} = \cch\).
\end{description}

\comment{A necessary and sufficient condition for position to be option}
\label{sec:orgheadline38}
The following lemma provides a sufficient condition for a position to satisfy (SG).
\begin{lem}
 \comment{Lem. minus-1}
\label{sec:orgheadline39}
\label{orgtarget31}
Let \(X\) be a position with \(\btower{X} > 0\).
Assume that \(\sg(Z) = \btower{Z}\) for each position \(Z\) with \(\size{\lambda(Z)} < \size{\lambda(X)}\).
If \(X\) has a descendant \(Y\) with \(\btower{Y} = \btower{X} - 1\),
then \(X\) satisfies (SG).
 
\end{lem}

\comment{Postpone the proof}
\label{sec:orgheadline40}
We postpone the proof of this lemma to the end of this section.

\comment{Show (SG2') assuming (A1) and (A2)}
\label{sec:orgheadline41}
Let \(X\) be a position with \(\btower{X} > 0\).
Our task now is to show that \(X\) has a descendant \(Y\)
with \(\btower{Y} = \btower{X} - 1\) assuming (A1) and (A2).
If \(\btower{X} = \size{\lambda(X)}\), then there is nothing to prove.
Suppose that \(\btower{X} < \size{\lambda(X)}\).
By (A2), the position \(X\) has a proper descendant \(Z\) with \(\btower{Z} \ge \btower{X}\).
By the induction hypothesis, \(\sg(Z) = \btower{Z} > \btower{X} - 1\),
and so \(Z\) has an option \(Y\) with \(\sg(Y) = \btower{Y} = \btower{X} - 1\).
This completes the proof.

\begin{rem}
 \comment{Rem.}
\label{sec:orgheadline42}
Suppose that \(p\) is a prime.
Then (A1) follows from Macdonald's result \cite{macdonald-degrees-1971} and the hook-length formula.
Indeed, let \(X\) be a position with \(\btower{X} = \size{\lambda(X)} > 0\),
and let \(\deg(\rho^X)\) be the degree of \(\rho^X\).
By the hook-length formula, 
\[
 \deg(\rho^X) = \sum_{Y} \deg(\rho^{Y}),
\]
where the sum is over all descendants \(Y\) of \(X\) with \(\ssize{\lambda(Y)} = \ssize{\lambda(X)} - 1\).
Since \(\btower{X} = \size{\lambda(X)}\) and \(p\) is a prime, \(\deg(\rho^X) \not \equiv 0 \pmod{p}\) by Macdonald's result,
and so \(\deg(\rho^{Y}) \not \equiv 0 \pmod{p}\) for some descendant \(Y\) of \(X\) with \(\ssize{\lambda(Y)} = \ssize{\lambda(X)} - 1\).
Thus \(\btower{Y} = \btower{X} - 1\).
 
\end{rem}

\subsection{\(p^\ccH\)-Options}
\label{sec:orgheadline59}
\label{orgtarget29}
\comment{Show (A1) at the end}
\label{sec:orgheadline44}
We show (A1) at the end of this subsection using \(p^\ccH\)-options.

\comment{Introduce a total order}
\label{sec:orgheadline45}
To define \(p^H\)-options, we first introduce a total order.
Let \((\alpha_L)_{L \in \NN}\) and \((\beta_L)_{L \in \NN}\) be two non-negative integer sequences with finitely many nonzero terms.
Suppose that \((\alpha_L)_{L \in \NN} \neq (\beta_L)_{L \in \NN}\).
Let \(N = \max \Set{L \in \NN : \alpha_L \neq \beta_L}\).
If \(\alpha_N < \beta_N\), then we write
\begin{equation}
\label{orgtarget32}
 (\alpha_L)_{L \in \NN} \prec (\beta_L)_{L \in \NN}.
\end{equation}

\comment{Order of a position}
\label{sec:orgheadline46}
We next define the (\(p\)-adic) order of a position.
For a non-terminal position \(X\) (that is, \(X\) has an option), 
the \emph{order} \(\ord(X)\) of \(X\) is defined by
\begin{equation}
\label{orgtarget33}
 \ord(X) = \min \Set{L \in \NN : \tower{X}[L] \neq 0}.
\end{equation}
If \(X\) is a terminal position, then we define \(\ord(X) = \infty\).
For example, \(\ord(\Set{x}) = \ord(x)\) for each \(x \in \NN\).

\begin{dfn}[\(p^\ccH\)-options and \(p^*\)-options]
 \comment{Dfn. [\(p^\ccH\)-options and \(p^*\)-options]}
\label{sec:orgheadline47}

Let \(X\) be a position with order \(M\) and \(Y\) its option \(X \cup \Set{x - p^\ccH} \setminus \Set{x}\).
The position \(Y\) is called a \emph{\(p^\ccH\)-option} of \(X\) if it has the following two properties:
\begin{enumerate}
\item \(\tower{Y}[L] \equiv \tower{X}[L] - 1 \pmod{p}\) for \(\ccH \le L \le M\).
\item \(\tower{Y}[\ge M+1] \succeq \tower{X}[\ge M+1]\).
\end{enumerate}
A \(p^\ccH\)-option of \(X\) is called a \emph{\(p^*\)-option} of \(X\) if \(\ccH = M\).
 
\end{dfn}

\begin{lem}
 \comment{Lem. existence-nondecreasing-option}
\label{sec:orgheadline48}
\label{orgtarget34}
Every non-terminal position has a \(p^*\)-option.
 
\end{lem}

\comment{Postpone the proof}
\label{sec:orgheadline49}
The proof of this lemma is deferred to Section \ref{orgtarget35}.

The next lemma is used to determine \(\tower{X}\).
\begin{lem}
 \comment{Lem. tower-lambda}
\label{sec:orgheadline50}
\label{orgtarget36}
Let \(X\) be a position.
Then \(\tower{X} \preceq \size{\lambda(X)} (= (\pexp{\size{\lambda(X)}}[0][p], \pexp{\size{\lambda(X)}}[1][p], \ldots))\).
In particular, if \(\tower{X}[\ge N] \succeq \pexp{\size{\lambda(X)}}[\ge N]\) for some \(N \in \NN\),
then \(\tower{X}[\ge N] = \pexp{\size{\lambda(X)}}[\ge N]\).
 
\end{lem}

\begin{proof}
 \comment{Proof.}
\label{sec:orgheadline51}
If \(\tower{X} \succ \size{\lambda(X)}\), then \(\sum_{L \in \NN} \tower{X}[L] p^L > \size{\lambda(X)}\),
which contradicts (\ref{orgtarget19}).
 
\end{proof}

\comment{Proof of (A1) for order 0}
\label{sec:orgheadline52}
We now prove (A1) for \(\ord(X) = 0\).
Since \(\ord(X) = 0\), Lemma \ref{orgtarget34} implies that \(X\) has a \(p^0\)-option \(Y\).
It suffices to show that \(\btower{Y} = \btower{X} - 1\). 
By the definition of \(p^0\)-options, 
\(\size{\lambda(Y)} = \size{\lambda(X)} - 1\).
Since \(\tower{X} = \btower{X}\), we have \(0 \neq \tower{X}[0] = \btower{X}[0] = \pexp{\ssize{\lambda(X)}}[0]\). This shows that
\[
 \size{\lambda(Y)} = \size{\lambda(X)} - 1 = (\tower{X}[0] - 1, \tower{X}[1], \tower{X}[2], \ldots).
\]
Since \(\tower{Y}[\ge 1] \succeq \tower{X}[\ge 1] = \pexp{\ssize{\lambda(Y)}}[\ge 1]\),
it follows from Lemma \ref{orgtarget36} that \(\tower{Y}[\ge 1] = \pexp{\size{\lambda(Y)}}[\ge 1]\), and therefore \(\btower{Y} = \size{\lambda(Y)} = \btower{X} - 1\).

\comment{A equivalence relation of positions}
\label{sec:orgheadline53}
To prove (A1) for \(\ord(X) > 0\),
we need \(p^0\)-options.
While every non-terminal position has a \(p^*\)-option, it does not necessarily have a \(p^0\)-option.
To state a sufficient condition for a position to have a \(p^0\)-option,
we introduce an equivalence relation on positions.
Let \(X\) and \(X'\) be two positions, and let \(N\) be a non-negative integer.
We write 
\begin{equation}
\label{orgtarget37}
 X \equiv X' \pmod{p^N}
\end{equation}
if \(\ssize{X_R} = \ssize{X'_R}\) for every \(R \in \ZZ_p^N\).
For example, if \(x, x' \in \NN\), then \(\set{x} \equiv \set{x'} \pmod{p^N}\) if and only if \(x \equiv x' \pmod{p^N}\).
By definition, this relation has the following properties.
\begin{lem}
 \comment{Lem.equivalent-then-sametower}
\label{sec:orgheadline54}
\label{orgtarget38}
Suppose that \(X\) and \(X'\) are two positions satisfying \(X \equiv X' \pmod{p^N}\) for some \(N \in \NN\).
Then \(X \equiv X' \pmod{p^L}\) for \(0 \le L \le N\), and \(\tower{X}[<N] = \tower{X'}[<N]\).
 
\end{lem}

\comment{A sufficient condition for a position to have 0-option}
\label{sec:orgheadline55}
The following lemma gives a sufficient condition for a position to have a \(p^0\)-option.
\begin{lem}
 \comment{Lem. borrow}
\label{sec:orgheadline56}
\label{orgtarget26}
Let \(X\) be a non-terminal position with order \(M\). Choose \(n \in \NN\) so that the size of \(\pshift{X}{\ccn}\) defined in (\ref{orgtarget23}) is a multiple of \(p^M\).
If 
\begin{description}
\item[{\textnormal{(P0)}}] \(\quad (\pshift{X}{\ccn})_{\ccs - 1} \not \equiv (\pshift{X}{n})_\ccs \pmod{p^M}  \quad \tforsome \ccs \in \ZZ_p\),
\end{description}
then \(X\) has a \(p^0\)-option.
 
\end{lem}

\comment{Postpone the proof}
\label{sec:orgheadline57}
We postpone the proof of this lemma to Section \ref{orgtarget35}.

We show that the condition (P0) is independent of the choice of \(\ccn\),
that is, \((\pshift{X}{\ccn})_{\ccs - 1} \not \equiv (\pshift{X}{n})_\ccs \pmod{p^M}\) for some \(\ccs \in \ZZ_p\) if and only
if \((\pshift{X}{\ccn + p^M})_{\cct - 1} \not \equiv (\pshift{X}{n + p^M})_\cct \pmod{p^M}\) for some \(\cct \in \ZZ_p\).
We may assume that \(n = 0\).
Since every position with order 0 satisfies (P0), we may also assume that \(M > 0\).
Let \(\delta_x = 1\) if \(x = 0\), and \(\delta_x = 0\) otherwise.
We have \(\ssize{(\pshift{X}{p^\ccM})_{\ccR \oplus p^\ccM}} = \ssize{X_\ccR} + \delta_{(\ccR_\ccM) + 1}\) for each \(\ccR \in \ZZ_p^{M + 1}\)
because \(\pexp{(x + p^M)}[< M + 1] = \pexp{x}[< M + 1] \oplus p^M\) for each \(x \in \NN\).
Note that \((R \ominus 1)_\ccM = R_\ccM\) since \(\ccM > 0\).
This implies that
\[
 \msize{(\pshift{X}{p^M})_{R \oplus p^M}}  - \msize{(\pshift{X}{p^M})_{R \oplus p^M \ominus 1}} = \msize{X_{\ccR}}  - \msize{X_{\ccR \ominus 1}}.
\]
Therefore (P0) is independent of the choice of \(\ccn\).

\comment{Proof of A1}
\label{sec:orgheadline58}
\vspace{1em}

\noindent
\emph{\textbf{proof of (A1) assuming Lemmas \ref{orgtarget34} and \ref{orgtarget26}.}}
Let \(M = \ord(X)\).
We have shown (A1) for \(M = 0\). Suppose that \(M > 0\).
We may assume that \(\ssize{X} \equiv 0 \pmod{p^M}\).

We show that \(X\) satisfies (P0).
Assume that \(X\) does not satisfy (P0).
Then by Lemma \ref{orgtarget38}, \(\tower{X_{\ccr - 1}}[M - 1] = \tower{X_{\ccr}}[M - 1]\) for each \(r \in \ZZ_p\).
Thus 
\[
 0 \neq \tower{X}[M] = \sum_{\ccr \in \ZZ_p} \tower{X_\ccr}[M - 1] = p \tower{X_0}[M - 1] \ge p,
\]
which contradicts \(\tower{X} = \btower{X}\).
Hence \(X\) satisfies (P0).

The position \(X\) has a \(p^0\)-option \(Y\) by Lemma \ref{orgtarget26}.
Since 
\[
 \size{\lambda(Y)} = \size{\lambda(X)} - 1 = (\underbrace{p - 1, \ldots, p - 1}_M, \tower{X}[M] - 1, \tower{X}[M + 1], \tower{X}[M + 2], \ldots)  
\]
and \(\tower{Y}[L] \equiv \tower{X}[L] - 1 \equiv \pexp{\size{\lambda(Y)}}[L] \pmod{p}\) for \(0 \le L \le M\),
it follows from Lemma \ref{orgtarget36} that
\(\btower{Y} = \size{\lambda(Y)} = \size{\lambda(X)} - 1 = \btower{X} - 1\).

\qed
\subsection{Peak digits}
\label{sec:orgheadline74}
\label{orgtarget30}
\comment{Intro}
\label{sec:orgheadline60}
It remains to prove (A2).
The idea of its proof is to focus on peak digits.
To define peak digits, we introduce \(p^*\)-descendants.
\begin{dfn}[\(p^*\)-paths and \(p^*\)-descendants]
 \comment{Dfn. [\(p^*\)-paths and \(p^*\)-descendants]}
\label{sec:orgheadline61}
Let \(n \in \NN\).
Let \((X^0, X^1, \ldots, X^\ccn)\) be a position sequence.
If \(X^{i+1}\) is a \(p^*\)-option of \(X^i\) for \(0 \le i \le n - 1\),
then this sequence is called a \emph{\(p^*\)-path} from \(X^0\) to \(X^\ccn\),
and \(X^\ccn\) is called a \emph{\(p^*\)-descendant} of \(X^0\).
 
\end{dfn}

\begin{dfn}[peak digits]
 \comment{Dfn. [peak digits]}
\label{sec:orgheadline62}
Let \(X\) be a position.
The \emph{peak digit} \(\pk(X)\) of \(X\) is defined by
\[
 \pk(X) = \max \Set{ \sL(\tower{X}, \tower{Y}) :
 Y \text{\ is a \(p^*\)-descendant of\ } X \text{\ with\ } \tower{Y} \succ \tower{X} },
\]
where $\max \emptyset = -1$ and
\[
 \sL(\tower{X}, \tower{Y}) = \max \Set{L \in \NN : \tower{X}[L] \neq \tower{Y}[L]}.
\]
 
\end{dfn}

\comment{Properties of peak digits}
\label{sec:orgheadline63}
It follows from (\ref{orgtarget19}) that if \(\pk(X) > -1\), then \(\pk(X) > \ord(X) \ge 0\).
Peak digits also have the following properties.
\begin{lem}
 \comment{Lem. peak-digit}
\label{sec:orgheadline64}
\label{orgtarget39}
If \(Y\) is a \(p^*\)-option of a position \(X\), then the following assertions hold:
\begin{enumerate}
\renewcommand{\labelenumi}{\textnormal{(\arabic{enumi})}}
\item $\tower{Y}[\ge N] = \tower{X}[\ge N]$, where $N = \max{\Set{\pk(X), \ord(X)}} + 1$.
\item $\pk(Y) \le \pk(X)$.
\end{enumerate}
 
\end{lem}
\begin{proof}
 \comment{Proof.}
\label{sec:orgheadline65}
(1)\,
Since \(Y\) is a \(p^*\)-option of \(X\), it follows that \(\tower{Y}[\ge (\ord(X) + 1)] \succeq \tower{X}[\ge (\ord(X) + 1)]\).
In addition, by the definition of peak digits, \(\tower{Y}[\ge (\pk(X) + 1)] \preceq \tower{X}[\ge (\pk(X) + 1)]\).
Thus \(\tower{Y}[\ge N] = \tower{X}[\ge N]\).

\vspace{0.8em}
\noindent
(2)\,
Let \(K = \pk(Y)\).
If \(K = -1\), then it is clear.
Suppose that \(K > -1\). 
Then \(Y\) has a \(p^*\)-descendant \(Z\) such that
\(\tower{Z} \succ \tower{Y}\) and  \(\sL(\tower{Z}, \tower{Y}) = K\).
Since \(K > \ord(Y) \ge \ord(X)\), it follows that
\[
 \tower{Z}[\ge K] \succ \tower{Y}[\ge K] \succeq \tower{X}[\ge K].
\]
This implies that \(\tower{Z} \succ \tower{X}\) and \(\sL(\tower{Z}, \tower{Y}) \ge K\).
Since \(Z\) is also a \(p^*\)-descendant of \(X\), we have \(\pk(X) \ge K = \pk(Y)\).
 
\end{proof}

\begin{cor}
 \comment{Cor.}
\label{sec:orgheadline66}
\label{orgtarget40}
Let \(n\) be a positive integer.
If \((X^0, \ldots, X^{\ccn})\) is a \(p^*\)-path and \(N\) is at least \(\max \set{\pk(X^{0}), \ord(X^{\ccn - 1})} + 1\),
then \(\tower{X^i}[\ge N] = \tower{X^0}[\ge N]\) for \(0 \le i \le \ccn\).
 
\end{cor}

\begin{proof}
 \comment{Proof.}
\label{sec:orgheadline67}
For \(0 \le i \le \ccn - 1\), let \(N^i = \max\Set{\pk(X^{i}), \ord(X^{i})} + 1\).
It follows from Lemma \ref{orgtarget39} that \(\tower{X^{i + 1}}[\ge N^i] = \tower{X^i}[\ge N^i]\) and
\[
 \pk(X^0) \ge \pk(X^1) \ge \cdots \ge \pk(X^{\ccn - 1}).
\]
Since
\[
  \ord(X^0) \le \ord(X^1) \le \cdots \le \ord(X^{\ccn - 1}),
\]
we have \(N \ge N^i\) for \(0 \le i \le \ccn - 1\), and the corollary follows.
 
\end{proof}

\comment{Ensure that X has a p*-descendant satisfying P0}
\label{sec:orgheadline68}
The following result is the key to proving Theorem 1.1'.
\begin{lem}
 \comment{Lem. borrowable-position}
\label{sec:orgheadline69}
\label{orgtarget41}
Let \(X\) be a position with peak digit \(K\).
If \(K\) is positive, then \(X\) has a descendant \(Y\) with the following four properties:
\begin{enumerate}
\item \(\ord(Y) = K\).
\item Either \(\tower{X}[K] < \tower{Y}[K] < p\;\) or \(\; \tower{Y}[K] = p\).
\item \(\tower{Y}[\ge K + 1] = \tower{X}[\ge K + 1]\).
\item \(Y\) satisfies (P0).
\end{enumerate}
 
\end{lem}

\comment{Deferred}
\label{sec:orgheadline70}
The proof of this lemma is deferred to Section \ref{orgtarget42}.

The next lemma determines \(\msg(X)\) for certain positions \(X\).
\begin{lem}
 \comment{Lem. esp-p0}
\label{sec:orgheadline71}
\label{orgtarget43}
Let \(X\) be a position with the following three properties:
\begin{enumerate}
\item \(\tower{X}[M] = p\), where \(M = \ord(X)\).
\item \(\tower{X}[\ge M + 1] = \btower{X}[\ge M + 1] (= (\btower{X}[M + 1], \btower{X}[M + 2], \ldots))\).
\item \(X\) satisfies (P0).
\end{enumerate}
Then \(X\) has a \(p^0\)-option \(Y\) with \(\btower{Y} = \size{\lambda(X)} - 1\).
In particular, if \(\btower{Y} = \sg(Y)\), then \(\msg(X) = \size{\lambda(X)} - 1\).
 
\end{lem}

\comment{Proof esp-p0}
\label{sec:orgheadline72}
\noindent
\emph{\textbf{proof assuming Lemma \ref{orgtarget26}.}}
The position \(X\) has a \(p^0\)-option \(Y\) by Lemma \ref{orgtarget26}.
Since \(\size{\lambda(X)} = p^{M + 1} + \sum_{L \ge M + 1} \tower{X}[L]p^L\) and \(\tower{X}[\ge M + 1] = \btower{X}[\ge M + 1]\),
we have
\[
 \size{\lambda(Y)} = \size{\lambda(X)} - 1 = (\underbrace{p - 1, \ldots, p - 1}_{M + 1}, \tower{X}[M + 1], \tower{X}[M + 2], \ldots).
\]
Thus \(\tower{Y}[\ge M + 1] = \pexp{\ssize{\lambda(Y)}}[\ge M + 1]\) by Lemma \ref{orgtarget36}.
Furthermore, since \(\tower{Y}[L] \equiv \tower{X}[L] - 1 \equiv p - 1 \pmod{p}\) for \(0 \le L \le M\),
we find that \(\tower{Y} = \ssize{\lambda(Y)} = \ssize{\lambda(X)} - 1\).

\qed

\comment{Proof Theorem 1.1}
\label{sec:orgheadline73}
\vspace{1em}

\noindent
\emph{\textbf{proof of Theorem 1.1' assuming Lemmas \ref{orgtarget31}, \ref{orgtarget34}, \ref{orgtarget26}, and \ref{orgtarget41}.}}
It remains to show (A2).
By the induction hypothesis, we may assume that \(\sg(X_r) = \btower{X_r} = \ssize{\lambda(X_r)}\) for each \(r \in \ZZ_p\). 

Since \(\btower{X} < \size{\lambda(X)}\), we see that \(\tower{X}[L] \ge p\) for some \(L \in \NN\).
Let
\[
 N = \max \Set{L \in \NN : \tower{X}[L] \ge p}.
\]
We divide into two cases.

\noindent
\textbf{Case 1: \(N > 0\).} 
Since \(\btower{X}[N] < p \le \tower{X}[N] = \sum_{\ccr \in \ZZ_p} \tower{X_\ccr}[N - 1]\), there exist \(s^0, \ldots s^{p - 1} \in \ZZ_p\)
such that \(\sum_{\ccr \in \ZZ_p} s^\ccr = \btower{X}[N]\) and \(s^\ccr \le \tower{X_\ccr}[N - 1] = \btower{X_\ccr}[N - 1]\) for each \(\ccr \in \ZZ_p\).
Since \(\sg(X_\ccr) = \btower{X_\ccr}\), the position \(X_\ccr\) has a descendant \(Y_\ccr\) such that
\[
 \tower{Y_\ccr}[L] = \begin{cases}
 s^\ccr & \tif L = N - 1,\\
 \tower{X_\ccr}[L] & \tif L \neq N - 1.\\
 \end{cases}
\]
Let \(Y = [Y_\ccr]_{\ccr \in \ZZ_p}\) defined in (\ref{orgtarget14}).
Then \(\btower{Y} = \btower{X}\) and \(Y \neq X\).

\noindent
\textbf{Case 2: \(N = 0\).}
Let \(K\) be the peak digit of \(X\).

Suppose that \(K = -1\).
By Lemma \ref{orgtarget39}, every \(p^*\)-option \(Y\) of \(X\) satisfies \(\tower{Y} = (\tower{X}[0] - 1, \tower{X}[1], \tower{X}[2], \ldots)\).
This implies that \(X\) has a desired descendant.

Suppose that \(K > -1\).
Then \(X\) has a descendant \(Y\) satisfying the conditions in Lemma \ref{orgtarget41}.
If \(\tower{Y}[K] < p\), then \(\btower{Y} > \btower{X}\).
Suppose that \(\tower{Y}[K] = p\).
Since \(Y\) satisfies (P0), it follows from Lemma \ref{orgtarget43} that
\[
 \msg(Y) = \size{\lambda(Y)} - 1 =  (\underbrace{p - 1, \ldots, p - 1}_{K + 1}, \tower{X}[K + 1], \tower{X}[K + 2], \ldots) \ge \btower{X}.
\]

\qed

\subsection{Proof of Proposition \ref{orgtarget10}}
\label{sec:orgheadline81}
\label{orgtarget44}
Before proving Proposition \ref{orgtarget10}, we present an elementary property of \(p^*\)-options.
\begin{lem}
 \comment{Lem. XrYr}
\label{sec:orgheadline75}
\label{orgtarget45}
Let \(X\) be a position whose order \(M\) is positive.
Choose \(s \in \ZZ_p\) so that \(\ord(X_s) = M - 1\).
Let \(Y_s\) be a \(p^*\)-option of \(X_s\), and let \(Y_r = X_r\) for each \(r \in \ZZ_p \setminus \set{s}\).
Then \([Y_\ccr]_{\ccr \in \ZZ_p}\) is a \(p^*\)-option of \(X\).
 
\end{lem}

\begin{proof}
 \comment{Proof.}
\label{sec:orgheadline76}
Let \(Y = [Y_\ccr]_{\ccr \in \ZZ_p}\).
Since \(Y_s\) is a \(p^{(M - 1)}\)-option of \(X_s\),
we have 
\[
 \tower{Y}[M] = \sum_{r \in \ZZ_p} \tower{Y_r}[M - 1] \equiv \sum_{r \in \ZZ_p} \tower{X_r}[M - 1] - 1 \equiv \tower{X}[M] - 1\pmod{p}
\]
and
\[
 \tower{Y}[\ge M+1] = \sum_{r \in \ZZ_p} \tower{Y_r}[\ge M] \succeq \sum_{r \in \ZZ_p} \tower{X_r}[\ge M] = \tower{X}[\ge M + 1].
\]
 
\end{proof}

\comment{Connect}
\label{sec:orgheadline77}
We now proceed to prove Proposition \ref{orgtarget10}.
\begin{lem}
 \comment{Lem. sum-over-p}
\label{sec:orgheadline78}
\label{orgtarget46}
Let \(X\) be a position with order \(M\).
If  \(\pk(X) = -1\) and \(\tower{X}[M] \ge p + 1\),
then \(X\) has a \(p^*\)-descendant \(Y\) with the following three properties:
\begin{enumerate}
\item \(\tower{Y}[M] = p\).
\item \(\tower{Y}[\ge M + 1] = \tower{X}[\ge M + 1]\).
\item \(Y\) satisfies (P0).
\end{enumerate}
 
\end{lem}

\comment{Proof of lemma}
\label{sec:orgheadline79}
\noindent
\emph{\textbf{proof assuming Lemma \ref{orgtarget34}.}}
Suppose that \(M = 0\). 
Since \(\pk(X) = -1\), we see that \(X\) has a \(p^*\)-descendant \(Y\) with \(\tower{Y}[0] = p\) and \(\tower{Y}[\ge 1] = \tower{X}[\ge 1]\).
The position \(Y\) also satisfies (P0),
since otherwise  \(\tower{Y}[0] = 0\).

Suppose that \(M > 0\).
For each \(\ccr \in \ZZ_p\), let \(a^\ccr = \tower{X_\ccr}[\ccM - 1]\) .
We show that there exists \((b^0, \ldots, b^{p - 1}) \in \NN^p\) such that
\begin{enumerate}
\item \(\sum_{\ccr \in \ZZ_p} b^\ccr = p\),
\item \(b^\ccr \le a^\ccr\) for each \(\ccr \in \ZZ_p\),
\item \(b^\ccs \neq b^\cct\) for some \(\ccs, \cct \in \ZZ_p\).
\end{enumerate}
We may assume that \(\sum_{\ccr \in \ZZ_p} a^\ccr = p + 1\) and \(a^0 \ge \ldots \ge a^{p - 1}\).
Let
\[
 (b^0, \ldots, b^{p - 1}) = \begin{cases}
 (a^0 - 1, a^1, a^2, \ldots, a^{p - 1}) & \tif a^1 = 0, \\
 (a^0, a^1 - 1, a^2, \ldots, a^{p - 1}) & \tif a^1 \neq 0. \\
 \end{cases}
\]
Then \((b^0, \ldots, b^{p - 1})\) has the desired properties.

Since \(\pk(X) = -1\), it follows from Lemma \ref{orgtarget45}
that \(X_\ccr\) has a \(p^*\)-descendant \(Y_\ccr\) such that
\(\tower{Y_\ccr}[M - 1] = b^\ccr\) and \(\tower{Y_\ccr}[\ge M] = \tower{X_\ccr}[\ge M]\).
Let \(Y = [Y_\ccr]_{\ccr \in \ZZ_p}\).
Then \(Y\) is a \(p^*\)-descendant of \(X\) such that \(\tower{Y}[M] = p\) and \(\tower{Y}[\ge M + 1] = \tower{X}[\ge M + 1]\).
Moreover, since \(b^\ccs \neq b^\cct\) for some \(\ccs, \cct \in \ZZ_p\), Lemma \ref{orgtarget38} implies that \(Y\) also satisfies (P0).
\qed

\comment{Proof of Proposition \ref{orgtarget10}}
\label{sec:orgheadline80}
\vspace{1em}
\noindent
\emph{\textbf{proof of Proposition \ref{orgtarget10} assuming Lemmas \ref{orgtarget31}, \ref{orgtarget34}, \ref{orgtarget26}, and \ref{orgtarget41}.}}
Applying Theorem \ref{orgtarget6} to \(X_R\) for each \(R \in \ZZ_p^{N + 1}\),
we find that \(X\) has a descendant \(Y\) such that
\[
 \tower{Y}[\le N] = \tower{X}[\le N] \quad \tand \quad \tower{Y}[\ge N + 1] = \btower{X}[\ge N + 1].
\]
Let 
\[
 \ccg = (\underbrace{p - 1, \ldots, p - 1}_{N + 1}, \btower{X}[N + 1], \btower{X}[N + 2], \ldots)_p.
\]
Our goal is to find a descendant \(Z\) of \(Y\) with \(\msg(Z) \ge \ccg\).
Let \(K\) be the peak digit of \(Y\). We split into two cases.

\vspace{0.8em}
\noindent
\textbf{Case 1: \(K \ge N\).} 
The position \(Y\) has a descendant \(Z\) 
satisfying the conditions in Lemma \ref{orgtarget41}.
If \(\tower{Y}[K] < \tower{Z}[K] < p\), then \(\msg(Z) \ge \btower{Z} > \ccg\).
If \(\tower{Z}[K] = p\),
then \(\msg(Z) \ge \ccg\) by Lemma \ref{orgtarget43}.

\vspace{0.8em}
\noindent
\textbf{Case 2: \(K < N\).}
Let \(Y^0 = Y\). By repeatedly applying Lemma \ref{orgtarget34},
we obtain a \(p^*\)-path \((Y^0, \ldots, Y^\ccn)\) such that
\begin{description}
\item[{\textnormal{(R1)}}] \(\ord(Y^\ccn) \ge N\),
\item[{\textnormal{(R2)}}] \(\ord(Y^{\cch}) < N\) for \(0 \le \cch < \ccn\).
\end{description}
Corollary \ref{orgtarget40} yields \(\tower{Y^\ccn}[\ge N] = \tower{Y}[\ge N]\).
Since \(\pk(Y^\ccn) = -1\) and \(\tower{Y^\ccn}[N] = \tower{Y}[N] \ge p + 1\), it follows from Lemmas \ref{orgtarget46} and \ref{orgtarget43} that \(\msg(Y^\ccn) = \ccg\).

\qed

Let \(X^0\) be a position and \(N \in \NN\).
As we have seen in the proof of Proposition \ref{orgtarget10}, there exists a \(p^*\)-path \((X^0, \ldots, X^\ccn)\) satisfying (R1) and (R2).
We call \(X^n\) an \emph{\((N-1)\)-rounded descendant} of \(X^0\).

\subsection{Proof of Lemma \ref{orgtarget31}}
\label{sec:orgheadline83}
\label{orgtarget47}
\begin{proof}
 \comment{Proof.}
\label{sec:orgheadline82}
Let \(\ccn = \btower{X}\).
We divide into two cases.

\vspace{0.8em}
\noindent
\textbf{Case 1: \(\ccn \not \equiv \cch \pmod{p}\).}
Since \(\ord(\btower{X} - \btower{Y}) = \ord(1) = 0\), it follows from Lemma \ref{orgtarget22} that \(Y\) is an option of \(X\).
Suppose that \(\cch < \ccn - 1\).
By assumption, \(\sg(Y) = \btower{Y}\), and hence \(Y\) has an option \(Z\) with \(\sg(Z) = \btower{Z} = \cch\).
Since \(\ord(\ccn - \cch) = 0\), the position \(Z\) is also an option of \(X\) by Lemma \ref{orgtarget22}. 

\vspace{0.8em}
\noindent
\textbf{Case 2: \(\ccn  \equiv \cch \pmod{p}\).}
We first construct a descendant \(Z\) of \(X\) with \(\btower{Z} = \cch\).
Let \(N = \ord(\ccn - \cch)\).
Since \(\ccn \equiv \cch\pmod{p}\) and \(\ccn > \cch\), it follows that \(N > 0\) and 
\(\pexp{\ccn}[\ge 1] > \pexp{\cch}[\ge 1]\).
Let \(a^\ccr = \btower{X_\ccr}\) for each \(\ccr \in \ZZ_p\).
Then \(a^0 \oplus \cdots \oplus a^{p - 1} = \pexp{\ccn}[\ge 1] > \pexp{\cch}[\ge 1]\).
By Lemma \ref{orgtarget21}, there exists \((b^0, \ldots, b^{p - 1}) \in \NN^p\)
such that
\begin{enumerate}
\item \(\bigoplus_{\ccr \in \ZZ_p} b^\ccr = \pexp{\cch}[\ge 1]\),
\item \(b^\ccr \le a^\ccr\) for each \(\ccr \in \ZZ_p\),
\item \(N - 1 = \ord(\bigoplus_{\ccr \in \ZZ_p} a^\ccr - \bigoplus_{\ccr \in \ZZ_p} b^\ccr) = \min \Set{\ord(a^\ccr - b^\ccr) : \ccr \in \ZZ_p}\).
\end{enumerate}
Since \(\ssize{\lambda(X_\ccr)} < \ssize{\lambda(X)}\), we have \(\sg(X_\ccr) = \btower{X_\ccr} = a^\ccr \ge b^\ccr\). 
If \(b^\ccr < a^\ccr\), then let \(Z_\ccr\) be an option of \(X_\ccr\) with \(\sg(Z_\ccr) = \btower{Z_\ccr} = b^\ccr\). 
If \(b^\ccr = a^\ccr\), then let \(Z_\ccr = X_\ccr\).
Let \(Z = [Z_\ccr]_{\ccr \in \ZZ_p}\). Then \(\btower{Z} = \cch\).

We next show that \(Z\) is an option of \(X\).
To this end, we use the tuple representation.
Let \(X = (x^1, \ldots, x^\ccm), X_\ccr = (x^{\ccr, 1}, \ldots x^{\ccr, \ccm^\ccr}), Z = (z^1, \ldots, z^\ccm)\), and \(Z_\ccr = (z^{\ccr, 1}, \ldots, z^{\ccr, \ccm^\ccr})\) 
for each \(\ccr \in \ZZ_p\).
Since \(Z_\ccr\) is an option of \(X_\ccr\) when \(b^\ccr < a^\ccr\), it follows from Lemma \ref{orgtarget22} that
\[
 N - 1 = \min \Set{\ord(a^r - b^r) : \ccr \in \ZZ_p} = \min \Set{\ord(x^{\ccr, i} - z^{\ccr, i}) : 1 \le i \le \ccm^\ccr,\; \ccr \in \ZZ_p},
\]
and so \(\min \set{\ord(x^i - z^i) : 1 \le i \le \ccm} = N = \ord(n - h)\).
Therefore \(Z\) is an option of \(X\) by Lemma \ref{orgtarget22}.
 
\end{proof}
\section{\(p^H\)-Options}
\label{sec:orgheadline100}
\label{orgtarget35}
In this section, we show Lemmas \ref{orgtarget34}
and \ref{orgtarget26}.
\subsection{Proof of Lemma \ref{orgtarget34}.}
\label{sec:orgheadline91}
We first give a sufficient condition for an option to be a \(p^*\)-option.
\begin{lem}
 \comment{Lem. hook}
\label{sec:orgheadline85}
\label{orgtarget48}
Let \(X\) be a position with order \(M\) and \(Y\) its option \(X \cup \Set{x - p^\ccH} \setminus \Set{x}\).
Then
\begin{equation}
\label{orgtarget49}
 \hook{Y}[L] - \hook{X}[L] = \begin{cases}
   - p^{\ccH - L} & \tif L \le \ccH, \\
 \ssize{X_{\pexp{x}[<L]}} - \ssize{X_{\pexp{(x - p^\ccH)}[<L]}} - \delta_{\pexp{x}[\ccH]} \delta_{\pexp{x}[\ccH + 1]}\cdots \delta_{\pexp{x}[L - 1]} - 1 & \tif L \ge \ccH + 1.\\
 \end{cases}
\end{equation}
In particular, if \(\ccH = \ccM\) and
\begin{equation}
\label{orgtarget50}
 \msize{X_{\pexp{(x - p^\ccM)}[<L]}} + \delta_{\pexp{x}[\ccM]} \delta_{\pexp{x}[\ccM + 1]} \cdots \delta_{\pexp{x}[L - 1]} <  \msize{X_{\pexp{x}[<L]}} \quad \tforevery L \ge \ccM + 1, 
\end{equation}
then \(Y\) is a \(p^*\)-option of \(X\).
 
\end{lem}

\begin{proof}
 \comment{Proof.}
\label{sec:orgheadline86}
Let \(Z = X \setminus \Set{x}\).
By (\ref{orgtarget17}), we have
\[
 \hook{X}[L] - \hook{Z}[L] = \sum_{R \in \ZZ_p^L} \hook{X_R}[0] - \hook{Z_R}[0] = \hook{X_{\pexp{x}[<L]}}[0] - \hook{Z_{\pexp{x}[<L]}}[0]
 = \pexp{x}[\ge L] - (\size{X_{\pexp{x}[<L]}} - 1).
\]
Since \(Z  = Y \setminus \Set{x - p^\ccH}\), we also have 
\(\hook{Y}[L] - \hook{Z}[L] = \pexp{(x - p^\ccH)}[\ge L] - (\ssize{Y_{\pexp{(x - p^\ccH)}[<L]}} - 1)\).
Thus
\[
 \hook{Y}[L] - \hook{X}[L] = \ssize{X_{\pexp{x}[<L]}} - \ssize{Y_{\pexp{(x - p^\ccH)}[<L]}} + \pexp{(x - p^\ccH)}[\ge L] - \pexp{x}[\ge L].
\]
If \(L \le \ccH\), then \(\pexp{x}[<L] = \pexp{(x - p^\ccH)}[<L]\), and so
\(\hook{Y}[L] - \hook{X}[L] = -p^{\ccH - L}\).
Suppose that \(L \ge \ccH + 1\).
Then \(\ssize{Y_{\pexp{(x - p^\ccH)}[<L]}} = \ssize{X_{\pexp{(x - p^\ccH)}[<L]}} + 1\)
and \(\pexp{(x - p^\ccH)}[\ge L] - \pexp{x}[\ge L] = - \delta_{\pexp{x}[\ccH]} \delta_{\pexp{x}[\ccH + 1]}\cdots \delta_{\pexp{x}[L - 1]}\),
which gives (\ref{orgtarget49}).

Suppose that \(\ccH = \ccM\) and \(X\) satisfies (\ref{orgtarget50}).
Then \(\hook{Y}[L] - \hook{X}[L] \ge 0\) for every \(L \ge \ccM + 1\).
This shows that \(\hook{Y}[\ge \ccM + 1] \succeq \hook{X}[\ge \ccM + 1]\), and so \(\tower{Y}[\ge \ccM + 1] \succeq \tower{X}[\ge \ccM + 1]\).
Since \(\tower{Y}[\ccM] \equiv \tower{X}[\ccM] - 1 \pmod{p}\),
the position \(Y\) is a \(p^*\)-option of \(X\).
 
\end{proof}

\comment{Connect}
\label{sec:orgheadline87}
The following result provides a sufficient condition for a position to have a \(p^*\)-option.
\begin{lem}
 \comment{Lem. \(p^*\) step2}
\label{sec:orgheadline88}
\label{orgtarget51}
Let \(X\) be a position with order \(\ccM\).
Let \(\ccH\) and \(\ccN\) be non-negative integers with \(\ccH \le \ccM \le \ccN - 1\).
Suppose that there exists \(\ccS \in \ZZ_p^{\ccN}\) such that
\begin{equation}
\label{orgtarget52}
 \msize{X_{\ccS - p^{\ccH}}} + \delta_{\pexp{\ccS}[\ccH]}\delta_{\pexp{\ccS}[\ccH + 1]} \cdots \delta_{\pexp{\ccS}[\ccN - 1]} < \msize{X_{\ccS}}.
\end{equation}
Then \(X\) has an option \(Y\) such that \(\tower{Y}[\ge N] \succeq \tower{X}[\ge N]\) and \(Y = X \cup \set{x - p^{\ccH}} \setminus \set{x}\) for some \(x \in X\) with \(\pexp{x}[<\ccN] = \ccS\).
In particular, if \(\ccH = \ccM = \ccN - 1\), then \(Y\) is a \(p^*\)-option of \(X\).
 
\end{lem}

\begin{proof}
 \comment{Proof.}
\label{sec:orgheadline89}
We first construct \(x\) satisfying
\(\ssize{X_{\pexp{(x - p^H)}[<L]}} + \delta_{\pexp{x}[\ccH]}\delta_{\pexp{x}[\ccH + 1]} \cdots \delta_{\pexp{x}[L - 1]} < \ssize{X_{\pexp{x}[<L]}}\) for every \(L \ge N\).
Let \(\delta = \delta_{\pexp{\ccS}[\ccH]} \delta_{\pexp{\ccS}[\ccH + 1]} \cdots \delta_{\pexp{\ccS}[\ccN - 1]}\).
Recall that \((\ccS, \ccr) = (\ccS_0, \ldots, \ccS_{\ccN - 1}, \ccr) \in \ZZ_p^{\ccN + 1}\) for each \(\ccr \in \ZZ_p\). 
Since \(\ssize{X_{\ccS}} = \sum_{\ccr \in \ZZ_p} \ssize{X_{(\ccS, \ccr)}}\) and
\(\sum_{\ccr \in \ZZ_p} \ssize{X_{(\ccS, \ccr) - p^{\ccH}}}  = \sum_{\ccr \in \ZZ_p} \ssize{X_{(\ccS - p^{\ccH}, \ccr)}}\), we have
\[
 \sum_{\ccr \in \ZZ_p} \msize{X_{(\ccS, \ccr) - p^{\ccH}}} + \delta = \msize{X_{\ccS - p^{\ccH}}} + \delta
 < \msize{X_\ccS} = \sum_{\ccr \in \ZZ_p} \msize{X_{(\ccS, \ccr)}}.
\]
This implies that
\[
 \msize{X_{(\ccS, \pexp{\ccS}[\ccN]) - p^{\ccH}}} + \delta \delta_{\pexp{\ccS}[\ccN]} < \msize{X_{(\ccS, \pexp{\ccS}[\ccN])}}
\]
for some \(\pexp{\ccS}[\ccN] \in \ZZ_p\).
Continuing this process, we obtain \(\ccS_{\ccN}, \ccS_{\ccN + 1}, \ldots \in \ZZ_p\).
Let \(x = (\ccS_0, \ccS_1, \ldots)\).
Then \(\pexp{x}[<\ccN] = \ccS\) and \(\ssize{X_{\pexp{(x - p^H)}[<L]}} + \delta_{\pexp{x}[\ccH]}\delta_{\pexp{x}[\ccH + 1]} \cdots \delta_{\pexp{x}[L - 1]} < \ssize{X_{\pexp{x}[<L]}}\) for every \(L \ge N\).

Let \(Y = X \cup \set{x - p^H} \setminus \set{x}\).
We next show that \(Y\) is an option of \(X\),
that is, \(x \in X,\; x - p^{\ccH} \not \in X\), and \(x - p^{\ccH} \ge 0\).
Let \(L = \min \set{L \in \NN : \pexp{x'}[\ge L] = 0 \; \tforevery \; x' \in X}\).
Then \(\ssize{X_\ccR} \in \set{0, 1}\) for each \(\ccR \in \ZZ_p^{L}\).
Since \(\ssize{X_{\pexp{(x - p^{\ccH})}[<L]}} + \delta_{\pexp{x}[\ccH]} \delta_{\pexp{x}[\ccH + 1]}\cdots \delta_{\pexp{x}[L - 1]}  < \ssize{X_{\pexp{x}[<L]}}\),
we have \(\ssize{X_{\pexp{(x - p^{\ccH})}[<L]}} + \delta_{\pexp{x}[\ccH]} \delta_{\pexp{x}[\ccH + 1]}\cdots \delta_{\pexp{x}[L - 1]} = 0\) and \(\ssize{X_{\pexp{x}[<L]}} = 1\).
This implies that \(Y\) is an option of \(X\), and therefore \(\tower{Y}[\ge N] \succeq \tower{X}[\ge N]\) by Lemma \ref{orgtarget48}.
 
\end{proof}

\comment{Proof of \(p^*\)-option}
\label{sec:orgheadline90}
\vspace{1em}

\noindent
\emph{\textbf{proof of Lemma \ref{orgtarget34}.}}
Let \(X\) be a non-terminal position with order \(M\).
By Lemma \ref{orgtarget45}, we may assume that \(M = 0\).

We show that
\(\ssize{X_{\ccs - 1}} + \delta_{\ccs} < \ssize{X_{\ccs}}\) for some \(\ccs \in \ZZ_p\).
We may assume that \(\ssize{X_{0}} \ge \ssize{X_1} \ge \cdots \ge \ssize{X_{p - 1}}\).
If \(\ssize{X_0} \in \set{\ssize{X_{p-1}}, \ssize{X_{p - 1}} + 1}\),
then \(\tower{X}[0] = 0\).
This shows that \(\ssize{X_{p-1}} + 1 < \ssize{X_0}\).
Hence \(X\) has a \(p^0\)-option by Lemma \ref{orgtarget51}.

\qed

\subsection{Proof of Lemma \ref{orgtarget26}}
\label{sec:orgheadline93}
\begin{proof}
 \comment{Proof.}
\label{sec:orgheadline92}
If \(M = 0\), then Lemma \ref{orgtarget34} shows that \(X\) has a \(p^0\)-option.
Suppose that \(M > 0\).
We may assume that \(\size{X} \equiv 0 \pmod{p^M}\).

\vspace{0.8em}
\noindent
\textbf{Claim.} There exists \(\ccT \in \ZZ_p^{M+1}\) such that
\(\pexp{\ccT}[0] \neq 0\) and \(\ssize{X_{\ccT \ominus 1}}  <\ssize{X_{\ccT}}\).

\vspace{0.8em}
Assuming this claim for the moment, we complete the proof.

Since \(\pexp{\ccT}[0] \neq 0\), we have \(\ssize{X_{\ccT \ominus 1}} = \ssize{X_{\ccT - 1}} + \delta_{\ccT_0} \cdots \delta_{\ccT_M}\).
By Lemma \ref{orgtarget51}, \(X\) has an option \(Y\)
such that \(\tower{Y}[\ge M + 1] \succeq \tower{X}[\ge M + 1]\) and \(Y = X \cup \Set{x - 1} \setminus \Set{x}\)
for some \(x \in X\) with \(\pexp{x}[< M + 1] = \ccT\).

We show that \(Y\) is a \(p^0\)-option of \(X\).
Since \(\ord(X) = M\) and \(\ssize{X} \equiv 0 \pmod{p^M}\),
it follows that \(\ssize{X_R} = \ssize{X}/p^L\) when \(0 \le L \le M\) and \(R \in \ZZ_p^L\).
By Lemma \ref{orgtarget48}, \(\tower{Y}[0] - \tower{X}[0] \equiv -1 \pmod{p}\) and
\begin{align*}
 \tower{Y}[L] - \tower{X}[L] &\equiv \ssize{X_{\pexp{x}[<L]}} - \ssize{X_{\pexp{(x - 1)}[<L]}} - \delta_{\pexp{x}[0]} \cdots \delta_{\pexp{x}[L - 1]} - 1 \\
 &\equiv \ssize{X}/p^L - \ssize{X}/p^L - 1 \equiv - 1 \pmod{p} \ \tfor 0 < L \le M.
\end{align*}
Thus \(Y\) is a \(p^0\)-option of \(X\).

It remains to prove the claim.
Since \(X\) satisfies (P0), we find that \(\ssize{X_{\ccS \ominus 1}} \neq \ssize{X_\ccS}\) for some \(\ccS \in \ZZ_p^{M + 1}\).
Hence there exists \(s \in \ZZ_p\) such that
\(\ssize{X_{\ccS \ominus s \ominus 1}} < \ssize{X_{\ccS \ominus s}}\), since otherwise
\(|X_{\ccS \ominus 1}| > |X_\ccS| \ge |X_{\ccS \oplus 1}| \ge \cdots \ge |X_{\ccS \oplus (p - 1)}| = |X_{\ccS \ominus 1}|\).

Let \(U = \ccS \ominus s\) and \(\widehat{\ccU} = (\ccU_1, \ldots, \ccU_{M - 1})\).
We may assume that \(\pexp{U}[0] = 0\). 
We show that there exist \(\pexp{\ccT}[0], \pexp{\ccT}[M] \in \ZZ_p\) such that \(\pexp{\ccT}[0] \neq 0\) and
\begin{equation}
\label{orgtarget53}
 \msize{X_{(\pexp{\ccT}[0], \widehat{\ccU}, \pexp{\ccT}[M]) \ominus 1}} < \msize{X_{(\pexp{\ccT}[0], \widehat{\ccU}, \pexp{\ccT}[M])}},
\end{equation}
where \((\pexp{\ccT}[0], \widehat{\ccU}, \pexp{\ccT}[M]) = (\pexp{\ccT}[0], \pexp{\ccU}[1], \ldots, \pexp{\ccU}[M - 1], \pexp{\ccT}[M])\).
Since \(M > 0\), we have
\[
 \sum_{\ccr \in \ZZ_p} \msize{X_{(\ccU_0 \ominus 1, \widehat{\ccU}, \ccr)}} = \msize{X_{(\ccU_0 \ominus 1, \widehat{\ccU})}} = \msize{X}/p^M = \msize{X_{(\ccU_0, \widehat{\ccU})}} = \sum_{\ccr \in \ZZ_p} \msize{X_{(\ccU_0, \widehat{\ccU}, \ccr)}}.
\]
Since \(\ssize{X_{(\pexp{\ccU}[0] \ominus 1, \widehat{\ccU}, \pexp{\ccU}[M])}} = \ssize{X_{\ccU \ominus 1}} < \size{X_\ccU} = \ssize{X_{(\pexp{\ccU}[0], \widehat{\ccU}, \pexp{\ccU}[M])}}\),
we find that \(\ssize{X_{(\pexp{\ccU}[0] \ominus 1, \widehat{\ccU}, \pexp{\ccT}[M])}} > \ssize{X_{(\pexp{\ccU}[0], \widehat{\ccU}, \pexp{\ccT}[M])}}\) for some \(\pexp{\ccT}[M] \in \ZZ_p\).
As we have seen above, there exists \(\pexp{\ccT}[0] \in \ZZ_p\) such that \(\pexp{\ccT}[0] \neq  \pexp{\ccU}[0] = 0\) and (\ref{orgtarget53}) holds.
This completes the proof.
 
\end{proof}

\subsection{\(p^*\)-Paths}
\label{sec:orgheadline99}
\comment{Intro}
\label{sec:orgheadline94}
It remains to prove Lemma \ref{orgtarget41}.
To this end, we introduce congruent \(p^*\)-paths.

\comment{Definition of congruent}
\label{sec:orgheadline95}
Let \((X^0, \ldots, X^\ccn)\) and \((\alt{X}^0, \ldots, \alt{X}^{\ccn})\) be two \(p^*\)-paths. 
Let \(N\) be a non-negative integer.
These two paths said to be \emph{congruent modulo} \(p^N\)
if \(X^i \equiv \alt{X}^{i} \pmod{p^N}\) for \(0 \le i \le \ccn\).

\comment{A sufficient for congruent}
\label{sec:orgheadline96}
We give a sufficient condition for these paths to be congruent modulo \(p^N\).
Let \(X^{i + 1} = X^i \cup \set{x^i - p^{M^i}} \setminus \set{x^i}\) and \(\alt{X}^{i + 1} = \alt{X}^{i} \cup \set{\salt{x}^{i} - p^{\alt{M}^i}} \setminus \set{\salt{x}^{i}}\)
for \(0 \le i \le \ccn - 1\).
Let \(S^i = \pexp{x^i}[<N]\) and \(\alt{S}^{i}= \pexp{\salt{x}^{i}}[<N]\).
If \(S^i = S^i - p^{M^i}\), then \(\ssize{X_R^{i + 1}} = \ssize{X_R^i}\) for each \(R \in \ZZ_p^N\).
If \(S^i \neq S^i - p^{M^i}\), then
\begin{equation}
\label{orgtarget54}
 \msize{X_R^{i + 1}} = \msize{X_R^{i}} +
 \begin{cases}
 -1 & \tif R = S^i, \\
 1 & \tif R = S^i - p^{M^i}, \\
 0 & \tif R \in \ZZ_p^N \setminus \Set{S^i, S^i - p^{M^i}}.
 \end{cases}
\end{equation}
Thus those paths are congruent modulo \(p^N\) if the following two conditions hold:
\begin{enumerate}
\item \(X^0 \equiv \alt{X}^{0} \pmod{p^N}\),
\item \(S^{i} = \alt{S}^{i}\) and \(S^{i} - p^{M^i} = \alt{S}^{i} - p^{\alt{M}^{i}}\) for \(0 \le i \le \ccn - 1\).
\end{enumerate}

The next lemma gives a sufficient condition for the existence of a congruent \(p^*\)-path.
\begin{lem}
 \comment{Lem. congruent p* path}
\label{sec:orgheadline97}
\label{orgtarget55}
Let \((X^0, \ldots, X^\ccn)\) be a \(p^*\)-path
and \(\alt{X}^{0}\) a position with \(\alt{X}^{0} \equiv X^0 \pmod {p^N}\) for some \(N \in \NN\).
Suppose that \(\tower{X^i}[\ge N] = \tower{X^0}[\ge N]\) for \(0 \le i \le \ccn\).
Then there exists a \(p^*\)-path \((\alt{X}^{0}, \ldots, \alt{X}^{\ccn})\)
congruent to \((X^0, \ldots, X^{\ccn})\) modulo \(p^N\)
such that \(\ord(\alt{X}^{i}) = \ord(X^i)\) for \(0 \le i \le \ccn - 1\).
 
\end{lem}

\begin{proof}
 \comment{Proof.}
\label{sec:orgheadline98}
The proof is by induction on \(\ccn\).
If \(\ccn = 0\), then there is nothing to prove. 
Suppose that \(\ccn > 0\).
By the induction hypothesis, there exists a \(p^*\)-path \((\alt{X}^{0}, \ldots, \alt{X}^{\ccn - 1})\) 
congruent to \((X^0, \ldots, X^{\ccn - 1})\) modulo \(p^N\) such that \(\ord(\alt{X}^{i}) = \ord(X^i)\) for \(0 \le i \le \ccn - 2\).
Let \(X = X^{\ccn - 1},\; \alt{X} = \alt{X}^{\ccn - 1}\), and \(Y = X^\ccn = X \cup \Set{x - p^M} \setminus \Set{x}\).

We first show that \(\ord(\alt{X}) = \ord(X) = M\).
Since \(\alt{X} \equiv X \pmod{p^N}\), it follows from Lemma \ref{orgtarget38} that \(\tower{\alt{X}}[<N] = \tower{X}[<N]\).
Moreover, since \(\tower{X}[\ge N] = \tower{Y}[\ge N]\),
we have \(M < N\), and so \(\ord(\alt{X}) = \ord(X) = M\).

We next construct a \(p^*\)-option \(\alt{Y}\) of \(\alt{X}\) with \(\alt{Y} \equiv Y \pmod{p^N}\).
Since \(\hook{Y}[N] = \hook{X}[N]\), it follows from Lemma \ref{orgtarget48} that
\begin{align*}
\size{\alt{X}_{\pexp{(x - p^M)}[< N]}} + \delta_{\pexp{x}[M]} \cdots \delta_{\pexp{x}[N - 1]} + 1 &= \size{X_{\pexp{(x - p^M)}[< N]}} + \delta_{\pexp{x}[M]} \cdots \delta_{\pexp{x}[N - 1]} + 1 \\
 &= \msize{X_{\pexp{x}[< N]}}  = \msize{\alt{X}_{\pexp{x}[< N]}}.
\end{align*}
Lemma \ref{orgtarget51} implies that \(\alt{X}\) has an option \(\alt{Y}\) such that \(\tower{\alt{Y}}[\ge N] \succeq \tower{\alt{X}}[\ge N]\)
and \(\alt{Y} = \alt{X} \cup \set{\salt{x} - p^M} \setminus \set{\salt{x}}\) for some \(\salt{x} \in \alt{X}\) with \(\pexp{\salt{x}}[<N] = \pexp{x}[<N]\).
Since \(\alt{X} \equiv X \pmod{p^N}\), we have \(\alt{Y} \equiv Y \pmod{p^N}\).
It remains to verify \(\tower{\alt{Y}}[\ge \ccM + 1] \succeq \tower{\alt{X}}[\ge \ccM + 1]\).
Recall that 
\[
  \tower{Y}[\ge N] = \tower{X}[\ge N], \quad
  \tower{Y}[\ge M + 1] \succeq \tower{X}[\ge M + 1], \quad \tand \
  \tower{\alt{X}}[<N] = \tower{X}[<N].
\]
In addition, since \(\alt{Y} \equiv Y \pmod{p^N}\), it follows that \(\tower{\alt{Y}}[<N] = \tower{Y}[<N]\).
This shows that \(\tower{\alt{Y}}[\ge M + 1] \succeq \tower{\alt{X}}[\ge M + 1]\),
and therefore \(\alt{Y}\) is a \(p^*\)-option of \(\alt{X}\).
 
\end{proof}
\section{Proof of Lemma \ref{orgtarget41}}
\label{sec:orgheadline117}
\label{orgtarget42}
\subsection{Outline}
\label{sec:orgheadline106}
\comment{Setup}
\label{sec:orgheadline101}
Let \(X\) be a position whose peak digit \(K\) is positive.
Let \(Z\) be a \(p^*\)-descendant of \(X\) such that 
\(\tower{Z}[K] \ge \tower{Y}[K]\) for every \(p^*\)-descendant \(Y\) of \(X\).
Then Corollary \ref{orgtarget40} yields \(\tower{Z}[\ge K + 1] = \tower{X}[\ge K + 1]\).
As we have seen in Subsection \ref{orgtarget44}, \(Z\) has a \((K - 1)\)-rounded descendant \(Y\).
The choice of \(Z\) implies that \(K > \pk(Z) \ge \pk(Y)\).
Thus \(\pk(Y) = -1\) and \(\tower{Y}[\ge K] = \tower{Z}[\ge K]\).
If \(\tower{Y}[K] < p\), then \(Y\) has the stated properties.
If \(\tower{Y}[K] > p\), then \(Y\) has a desired descendant by Lemma \ref{orgtarget46}.
Suppose that \(\tower{Y}[K] = p\).
We may assume that \(\ssize{X} \equiv 0 \pmod{p^K}\) and \(Y\) does not satisfy (P0).

\comment{Goal is to find \alt{Y}}
\label{sec:orgheadline102}
Our goal is to find a \((K - 1)\)-rounded descendant \(\alt{Y}\) of \(X\) satisfying \(\tower{\alt{Y}}[K] = p\) and (P0).
Let \(\alt{Y}\) be a \((K - 1)\)-rounded descendant of \(X\) with \(\tower{\alt{Y}}[K] = p\).
The position \(\alt{Y}\) satisfies (P0) if there does not exist \(\ccS \in \ZZ_p^K\) such that
\begin{equation}
\label{orgtarget56}
 \tower{\alt{Y}_{\ccR}}[0] = \begin{cases}
 1 & \tif \ccR \in \Set{\ccS, \ccS \ominus 1, \ldots, \ccS \ominus (p - 1)}, \\
 0 & \tif \ccR \in \ZZ_p^K \setminus \Set{\ccS, \ccS \ominus 1, \ldots, \ccS \ominus (p - 1)}. \\
 \end{cases}
\end{equation}
Indeed, suppose that \(\alt{Y}\) does not satisfy (P0).
Then \(\tower{\alt{Y}_{\ccR \ominus 1}}[0] = \tower{\alt{Y}_\ccR}[0]\) for each \(\ccR \in \ZZ_p^K\), and so there exists \(\ccS \in \ZZ_p^K\) satisfying  (\ref{orgtarget56}).

\comment{p* Path}
\label{sec:orgheadline103}
Let \(X^0 = X\) and \(X^\ccn = Y\). 
Let \((X^0, \ldots, X^\ccn)\) be a \(p^*\)-path from \(X^0\) to \(X^\ccn\) through \(Z\),
that is, \(X^\cch = Z\) for some \(\cch\) with \(0 < \cch \le \ccn\).
We may assume that \(\tower{X^{\cch - 1}}[K] < p\).

We divide into two cases.

\comment{If some tower \(\ge\) 2}
\label{sec:orgheadline104}
\noindent
\textbf{Case 1: \(\tower{X^j_\ccS}[0] \ge 2\) for some \(j \in \Set{\cch, \cch + 1, \ldots, \ccn}\) and some \(\ccS \in \ZZ_p^K\).} 
We show that \(Z\) has a \((K - 1)\)-rounded descendant \(\alt{Y}\) such that \(\tower{\alt{Y}_\ccT}[0] \ge 2\) for some \(\ccT \in \ZZ_p^K\).
In particular, \(\alt{Y}\) satisfies (P0).

\vspace{0.8em}

\comment{If towers \(\le\) 2}
\label{sec:orgheadline105}
\noindent
\textbf{Case 2: \(\tower{X^i_\ccR}[0] \le 1\) for each \(i \in \Set{\cch, \cch + 1, \ldots, \ccn}\) and each \(\ccR \in \ZZ_p^K\).}
We show that \(X^{\cch - 1}\) has another \(p^*\)-option \(\alt{Z}\) such that 
\begin{align*}
 \bigl(\tower{\alt{Z}_\ccS}[0],\, \tower{\alt{Z}_{\ccT}}[0]\bigr) &= \bigl(\tower{Z_{\ccT}}[0],\, \tower{Z_{\ccS}}[0]\bigr) = (0, 1) \quad \tforsome \ccS, \ccT \in \ZZ_p^K, \\
 \tower{\alt{Z}_{\ccR}}[0]& = \tower{Z_{\ccR}}[0] \quad \tforeach \ccR \in \ZZ_p^K \setminus \Set{\ccS, \ccT}.
\end{align*}
Using Lemma \ref{orgtarget55}, we then show that
\(\alt{Z}\) has a \((K - 1)\)-rounded descendant \(\alt{Y}\) such that
\begin{align*}
 \bigl(\tower{\alt{Y}_\ccU}[0],\, \tower{\alt{Y}_{\ccV}}[0]\bigr) &= \bigl(\tower{Y_{\ccV}}[0],\, \tower{Y_{\ccU}}[0]\bigr) = (0, 1) \quad \tforsome \ccU, \ccV \in \ZZ_p^K, \\
 \tower{\alt{Y}_{\ccR}}[0]& = \tower{Y_{\ccR}}[0] \quad \tforeach \ccR \in \ZZ_p^K \setminus \Set{\ccU, \ccV}.
\end{align*}
In particular, \(\alt{Y}\) satisfies (P0).

\subsection{\((\tower{X_{S}}[0], \tower{X_{S - p^M}}[0])\) and \((\tower{Y_{S}}[0], \tower{Y_{S - p^M}}[0])\)}
\label{sec:orgheadline113}
Let \(X\) be a position and \(Y\) its \(p^*\)-option \(X \cup \set{x - p^M} \setminus \set{x}\).
Let \(L \in \NN\) and \(S = \pexp{x}[<L]\).
For each \(R \in \ZZ_p^L \setminus \set{S, S - p^M}\), we have \(X_R = Y_R\), and so \(\tower{X_R}[0] = \tower{Y_R}[0]\).
In this subsection, we examine \((\tower{X_{S}}[0], \tower{X_{S - p^M}}[0])\) and \((\tower{Y_{S}}[0], \tower{Y_{S - p^M}}[0])\).

We begin by presenting a lower bound for \(\pk(X)\).
\begin{lem}
 \comment{Lem. bound-peak-digits}
\label{sec:orgheadline107}
\label{orgtarget57}
Let \(X\) be a position with order \(M\),
and let \(N\) be an integer with \(N \ge M + 1\).
If \(\ssize{X_{\ccS - p^M}} + \delta_{\ccS_M} \delta_{\ccS_{M + 1}} \cdots \delta_{\ccS_{N - 1}} + 1 < \ssize{X_{\ccS}}\) for some \(\ccS \in \ZZ_p^N\),
then \(N \le \pk(X)\).
 
\end{lem}

\begin{proof}
 \comment{Proof.}
\label{sec:orgheadline108}
By Lemma \ref{orgtarget51},
\(X\) has an option \(Y\) such that \(\tower{Y}[\ge N] \succeq \tower{X}[\ge N]\)
and \(Y = X \cup \Set{x - p^M} \setminus \Set{x}\) for some \(x \in X\) with \(\pexp{x}[<N] = \ccS\).
Moreover, Lemma \ref{orgtarget48} implies that
\begin{align*}
 \hook{Y}[N] - \hook{X}[N] &= \msize{X_{\pexp{x}[<N]}} - \msize{X_{\pexp{(x - p^M)}[<N]}} - \delta_{x_M} \delta_{x_{M + 1}} \cdots \delta_{x_{N - 1}} - 1 > 0.
\end{align*}
Hence \(Y\) is a \(p^*\)-option of \(X\) with \(\tower{Y} \succ \tower{X}\) and \(\sL(\tower{Y}, \tower{X}) \ge N\).
Therefore \(N \le \pk(X)\).
 
\end{proof}

\begin{lem}
 \comment{Lem. over-two}
\label{sec:orgheadline109}
\label{orgtarget58}
Let \(X\) be a position
and \(Y\) its \(p^*\)-option \(X \cup \Set{x - p^M} \setminus \Set{x}\).
Let \(N\) be a non-negative integer with \(N \ge \max \set{M + 1, \pk(X)}\),
and let \(\ccS = \pexp{x}[<N]\).
If \(\tower{Y_{\ccS}}[0] + \tower{Y_{\ccS - p^M}}[0] \ge 2\),
then \(X\) has a \(p^*\)-option \(\alt{Y}\) such that
\begin{enumerate}
\item \(\alt{Y} \equiv Y \pmod{p^N}\),
\item \(\tower{\alt{Y}_{\ccS}}[0] \ge 2\) or \(\tower{\alt{Y}_{\ccS - p^M}}[0] \ge 2\).
\end{enumerate}
 
\end{lem}

\begin{proof}
 \comment{Proof.}
\label{sec:orgheadline110}
We may assume that \(\tower{Y_\ccS}[0] = \tower{Y_{\ccS - p^M}}[0] = 1\).
Then 
\begin{equation}
\label{orgtarget59}
 (Y_\ccS)_{(p)} = \pshift{\set{1}}{\ssize{Y_\ccS} - 1} \quad \tand \quad (Y_{\ccS - p^M})_{(p)} = \pshift{\set{1}}{\ssize{Y_{\ccS - p^M}} - 1}.
\end{equation}
Replacing \(X\) by \(\pshift{X}{p^M}\), if necessary, we may assume that \(\pex{x}{M} \neq 0\).

We first consider \(\ssize{Y_{\ccS}}\) and \(\ssize{Y_{\ccS - p^M}}\).
Since \(N \ge \max \set{\ccM + 1, \pk(X)}\), it follows from Lemma \ref{orgtarget39} that \(\tower{Y}[\ge N + 1] = \tower{X}[\ge N + 1]\).
Thus \(\tower{Y}[N] - \tower{X}[N] = \hook{Y}[N] - \hook{X}[N] \ge 0\).
Let \(\Delta = \tower{Y}[N] - \tower{X}[N]\).
By Lemma \ref{orgtarget48}, \(\ssize{X_\ccS} = \ssize{X_{\ccS - p^M}} + \Delta + 1\),
and therefore
\begin{equation}
\label{orgtarget60}
 \msize{Y_\ccS} = \msize{Y_{\ccS - p^M}} + \Delta - 1.
\end{equation}
Furthermore, since \(\tower{Y_{\ccS}}[0] + \tower{Y_{\ccS - p^M}}[0] = 2\) and \(\Delta \ge 0\),
it follows that \(\Delta \in \Set{0, 1, 2}\).

We next consider \(\ssize{Y_{(\ccS, \ccr)}}\) and \(\ssize{Y_{(\ccS - p^M, \ccr)}}\) for each \(\ccr \in \ZZ_p\). 
Lemma \ref{orgtarget57} yields \(\ssize{X_{(\ccS, \ccr)}} \le \ssize{X_{(\ccS, \ccr) - p^M}} + 1 = \ssize{X_{(\ccS - p^M, \ccr)}} + 1\),
and so 
\begin{equation}
\label{orgtarget61}
 \msize{Y_{(\ccS, \ccr)}} \le \msize{Y_{(\ccS - p^M, \ccr)}} + 1.
\end{equation}
Since \(\hook{Y}[N + 1] = \hook{X}[N + 1]\), we have
\(\ssize{X_{(\ccS, x_N)}} = \ssize{X_{(\ccS - p^M, x_N)}} + 1\),
and hence 
\begin{equation}
\label{orgtarget62}
 \msize{Y_{(\ccS, x_N)}} = \msize{Y_{(\ccS - p^M, x_N)}} - 1.
\end{equation}

Let \(\cct = \pexp{\ssize{Y_{\ccS - p^M}}}[0]\).
We split into three cases depending on \(\Delta\).

\noindent
\textbf{Case 1: \(\Delta = 0\).} 
We have \(\ssize{Y_{\ccS}} = \ssize{Y_{\ccS - p^M}} - 1\).
If \(p = 2\), then by (\ref{orgtarget59}), \(\ssize{Y_{(\ccS, \cct)}} = \ssize{Y_{(\ccS - p^M, \cct)}} - 2\) and
\(\ssize{Y_{(\ccS, \cct - 1)}} = \ssize{Y_{(\ccS - p^M, \cct - 1)}} + 1\), which contradicts (\ref{orgtarget62}).
Suppose that \(p > 2\). By (\ref{orgtarget59}),
\[
 \msize{Y_{(\ccS, \ccr)}} = \msize{Y_{(S - p^M, \ccr)}} + \begin{cases}
   -1 & \tif \ccr \in \Set{\cct, \cct - 2}, \\
   1 & \tif \ccr = \cct - 1,\\
   0 & \tif \ccr \in \ZZ_p \setminus \Set{\cct, \cct - 1, \cct - 2}.
 \end{cases}
\]
Thus \(x_N \in \set{\cct, \cct - 2}\) and \(\ssize{X_{(S, t - 1)}} = \ssize{X_{(S, t - 1) - p^M}} + 1\).
By Lemma \ref{orgtarget51}, \(X\) has an option \(\alt{Y}\) such that \(\tower{\alt{Y}}[\ge N + 1] \succeq \tower{X}[\ge N + 1]\) and \(\alt{Y} = X \cup \Set{\salt{x} - p^M} \setminus \Set{\salt{x}}\)
for some \(\salt{x} \in X\) with \(\pexp{\salt{x}}[< N+1] = (S, \cct - 1)\).
Since \(\pexp{\salt{x}}[N] = \cct - 1\), we find that
 \((\tower{\alt{Y}_{\ccS}}[0], \tower{\alt{Y}_{\ccS - p^M}}[0]) = (2, 0)\) if \(x_N = \cct\) and \((\tower{\alt{Y}_{\ccS}}[0], \tower{\alt{Y}_{\ccS - p^M}}[0]) = (0, 2)\) if \(x_N = \cct - 2\).
It remains to show that \(\alt{Y}\) is a \(p^*\)-option of \(X\) with \(\alt{Y} \equiv Y \pmod{p^N}\).
Since \(\hook{\alt{Y}}[N] - \hook{X}[N] = \ssize{X_S} - \ssize{X_{S - p^M}} - 1 = 0\),
we have \(\tower{\alt{Y}}[\ge N] \succeq \tower{X}[\ge N] = \tower{Y}[\ge N]\).
Since \(\pexp{\salt{x}}[<N] = \pexp{x}[<N]\), it follows that \(\alt{Y} \equiv Y \pmod{p^N}\), and so \(\tower{\alt{Y}}[< N] = \tower{Y}[< N]\).
Hence \(\tower{\alt{Y}}[\ge M + 1] \succeq \tower{Y}[\ge M + 1] \succeq \tower{X}[\ge M + 1]\). 

\noindent
\textbf{Case 2: \(\Delta = 1\).} 
We have \(\ssize{Y_{\ccS}} = \ssize{Y_{\ccS - p^M}}\), and so
\(\ssize{Y_{(\ccS, \ccr)}} = \ssize{Y_{(\ccS - p^M, \ccr)}}\) for each \(\ccr \in \ZZ_p\),
which is a contradiction.

\noindent
\textbf{Case 3: \(\Delta = 2\).} 
We have \(\ssize{Y_{\ccS}} = \ssize{Y_{\ccS - p^M}} + 1\).
If \(p = 2\), then by (\ref{orgtarget59}),
\(\ssize{Y_{(\ccS, \cct - 1)}} = \ssize{Y_{(\ccS - p^M, \cct - 1)}} + 2\), which contradicts (\ref{orgtarget61}).
Suppose that \(p > 2\). By (\ref{orgtarget59}), we have
\[
 \msize{Y_{(\ccS, \ccr)}} = \msize{Y_{(S - p^M, \ccr)}} + \begin{cases}
   -1 & \tif \ccr = \cct,\\
   1 & \tif \ccr \in \Set{\cct - 1, \cct + 1}, \\
   0 & \tif \ccr \in \ZZ_p \setminus \Set{\cct, \cct - 1, \cct + 1}.
 \end{cases}
\]
Thus \(x_N = \cct\).
By Lemma \ref{orgtarget51},
\(X\) has a \(p^*\)-option \(\alt{Y}\) such that 
\(\alt{Y} = X \cup \Set{\salt{x} - p^M} \setminus \Set{\salt{x}}\) for some \(\salt{x} \in X\) with
\(\pexp{\salt{x}}[<N + 1] = (S, t - 1)\).
It follows that \(\alt{Y} \equiv Y \pmod{p^N}\) and \(\tower{\alt{Y}_{\ccS}}[0] = 2\).
 
\end{proof}

\begin{lem}
 \comment{Lem. equal-one}
\label{sec:orgheadline111}
\label{orgtarget63}
Let \(X, Y, M, N\), and \(\ccS\) be as in Lemma \ref{orgtarget58}.
Let \(\Delta = \tower{Y}[N] - \tower{X}[N]\).
Suppose that \(\tower{Y_{\ccS}}[0] + \tower{Y_{\ccS - p^M}}[0] = 1\).
\begin{enumerate}
\renewcommand{\labelenumi}{\textnormal{(\arabic{enumi})}}
\item If $\Delta = 1$, then $X$ has a \(p^*\)-option $\alt{Y}$
 with $\alt{Y} \equiv Y \pmod{p^N}$ and $\bigl(\tower{\alt{Y}_\ccS}[0], \tower{\alt{Y}_{\ccS - p^M}}[0]\bigr) = \bigl(\tower{Y_{\ccS - p^M}}[0], \tower{Y_\ccS}[0]\bigr)$.
\item If $\Delta = 0$, then $\bigl(\tower{X_\ccS}[0], \tower{X_{\ccS - p^M}}[0]\bigr) = \bigl(\tower{Y_{\ccS - p^M}}[0], \tower{Y_\ccS}[0]\bigr)$.
\end{enumerate}
 
\end{lem}

\begin{proof}
 \comment{Proof.}
\label{sec:orgheadline112}
Suppose that \(\Delta = 1\).
Then \(\tower{X_\ccS}[0] = \tower{X_{\ccS - p^M}}[0] = 0\).
Replacing \(X\) by \(\pshift{X}{p^M}\), if necessary, we may assume that \(\pex{x}{M} \neq 0\).
Lemma \ref{orgtarget48} yields \(\ssize{X_\ccS} = \ssize{X_{\ccS - p^M}} + 2\). 
Let \(\cct = \pexp{\ssize{X_{\ccS - p^M}}}[0]\).
Since \((X_\ccS)_{(p)} = \pshift{\emptyset}{\ssize{X_\ccS}}\) and \((X_{\ccS - p^M})_{(p)} = \pshift{\emptyset}{\ssize{X_{\ccS - p^M}}}\),
we have
\[
 \msize{X_{(\ccS, \ccr)}} = \msize{X_{(\ccS - p^M, \ccr)}} + \begin{cases}
   1 & \tif \ccr \in \Set{\cct, \cct + 1},\\
  0 & \tif \ccr \in \ZZ_p \setminus \Set{\cct, \cct + 1}.\\
 \end{cases}
\]
If \(\pex{x}{N} = \cct\), then \(\tower{Y_{\ccS}}[0] = 1\).
If \(\pex{x}{N} = \cct + 1\), then \(\tower{Y_{\ccS - p^M}}[0] = 1\).
It follows from Lemma \ref{orgtarget51} that \(X\) has a desired \(p^*\)-option.

The proof for \(\Delta = 0\) is similar.
 
\end{proof}

\subsection{Proof of Lemma \ref{orgtarget41}}
\label{sec:orgheadline116}
\begin{proof}
 \comment{Proof.}
\label{sec:orgheadline114}
Let \(X, Y, Z, K\), and \((X^0, \ldots, X^\cch, \ldots, X^\ccn)\) be as in the beginning of this section.
Let \(W = X^{\cch - 1}\) and \(Z = W \cup \set{w - p^M} \setminus \set{w}\),
and let \(\ccS = \pexp{w}[<K]\).
Since \(\tower{W}[K] < \tower{Z}[K] = p\) and \(\ord(W) < \pk(W) = K\),
it follows from Lemma \ref{orgtarget58} that 
we may assume that \(\tower{Z_S}[0] + \tower{Z_{S - p^M}}[0] = 1\).
Hence \(\tower{W}[K] = p - 1\) and
\[
 \bigl(\tower{Z_{\ccS}}[0],\, \tower{Z_{\ccS - p^M}}[0]\bigr) \in \Set{(0, 1), (1, 0)}.
\]

By Lemma \ref{orgtarget63}, \(W\) has another \(p^*\)-option \(\alt{Z}\)
with \(\alt{Z} \equiv Z \pmod{p^K}\) such that
\begin{equation}
\label{orgtarget64}
\begin{split}
  \bigl(\tower{\alt{Z}_{\ccS}}[0],\, \tower{\alt{Z}_{\ccS - p^M}}[0]\bigr) &= \bigl(\tower{Z_{\ccS - p^M}}[0],\, \tower{Z_{\ccS}}[0]\bigr),\\
 \tower{\alt{Z}_{\ccR}}[0] &= \tower{Z_{\ccR}}[0] \quad \tforeach \ccR \in \ZZ_p^K \setminus \set{\ccS, \ccS - p^M}.
\end{split}
\end{equation}
Note that \(\tower{\alt{Z}}[K] = p\), and that \(\tower{\alt{Z}}[\ge K + 1] = \tower{X}[\ge K + 1]\) by Corollary \ref{orgtarget40}.
Let \(\alt{X}^\cch = \alt{Z}\).
Since \(\tower{X^i}[\ge K] = \tower{X^\cch}[\ge K]\) for \(\cch \le i \le \ccn\),
it follows from Lemma \ref{orgtarget55} that there exists a \(p^*\)-path \((\alt{X}^\cch, \ldots, \alt{X}^\ccn)\) congruent to \((X^\cch, \ldots, X^\ccn)\) modulo \(p^K\)
such that \(\ord(\alt{X}^i) \equiv \ord(X^i)\) for \(\cch \le i \le \ccn - 1\).

Let \(\alt{Y} = \alt{X}^\ccn\).
Then \(\alt{Y} \equiv Y \pmod{p^K}\) and \(\tower{\alt{Y}}[\ge K] = \tower{\alt{Z}}[\ge K]\).
Thus \(\tower{\alt{Y}}[K] = p,\; \ord(\alt{Y}) = K\), and \(\tower{\alt{Y}}[\ge K + 1] = \tower{X}[\ge K + 1]\).
It remains to show that \(\alt{Y}\) satisfies (P0).
For \(\cch \le i \le \ccn - 1\), let 
\[
 X^{i + 1} = X^i \cup \set{x^i - p^{M^{i}}} \setminus \set{x^i} \quad \tand \quad \alt{X}^{i + 1} = \alt{X}^i \cup \set{\salt{x}^i - p^{M^{i}}} \setminus \set{\salt{x}^i}.
\]
Then \(\pexp{x^i}[<K] = \pexp{\salt{x}^i}[<K]\).
Let \(\ccS^i = \pexp{x^i}[< K]\) for \(\cch \le i \le \ccn - 1\).
For each \(\ccR \in \ZZ_p^k \setminus \set{\ccS^i, \ccS^i - p^{M^i}}\),
we have \(X^{i+1}_\ccR = X^{i}_\ccR\) and \(\alt{X}^{i + 1}_\ccR = \alt{X}^i_\ccR\) .
By Lemma \ref{orgtarget58}, we may assume that
\[
 \bigl(\tower{X^{i}_{\ccS^i}}[0], \tower{X^{i}_{\ccS^i - p^{M^i}}}[0]\bigr),\, \bigl(\tower{\alt{X}^i_{\ccS^i}}[0], \tower{\alt{X}^i_{\ccS^i - p^{M^i}}}[0]\bigr) \in \Set{(0, 0), (0, 1), (1, 0)} \quad \text{for\ } \cch \le i \le \ccn - 1.
\]
It follows from Lemma \ref{orgtarget63} that 
\begin{equation}
\label{orgtarget65}
\begin{split}
 \bigl(\tower{X^{i + 1}_{\ccS^{i}}}[0],\, \tower{X^{i + 1}_{\ccS^i - p^{M^i}}}[0]\bigr) &= \bigl(\tower{X^{i}_{\ccS^i - p^{M^i}}}[0],\, \tower{X^{i}_{\ccS^i}}[0]\bigr),\\
 \bigl(\tower{\alt{X}^{i + 1}_{\ccS^i}}[0],\, \tower{\alt{X}^{i + 1}_{\ccS^i - p^{M^i}}}[0]\bigr) &= \bigl(\tower{\alt{X}^i_{\ccS^i - p^{M^i}}}[0],\, \tower{\alt{X}^i_{\ccS^i}}[0]\bigr).
\end{split}
\end{equation}
By (\ref{orgtarget64}) and (\ref{orgtarget65}),
\begin{align*}
 \bigl(\tower{\alt{Y}_\ccU}[0],\, \tower{\alt{Y}_{\ccV}}[0]\bigr) &= \bigl(\tower{Y_{\ccV}}[0],\, \tower{Y_{\ccU}}[0]\bigr) = (0, 1) \quad \tforsome \ccU, \ccV \in \ZZ_p^K, \\
 \tower{\alt{Y}_{\ccR}}[0]& = \tower{Y_{\ccR}}[0] \quad \tforeach \ccR \in \ZZ_p^K \setminus \Set{\ccU, \ccV}.
\end{align*}
Therefore \(\alt{Y}\) satisfies the condition (P0). 
 
\end{proof}

\comment{bibliography}
\label{sec:orgheadline115}
\bibliography{library}
\end{document}